\documentclass[a4paper,12pt]{amsart}

\usepackage[foot]{amsaddr}
\usepackage[T1]{fontenc}
\usepackage[utf8]{inputenc}
\usepackage{palatino}
\usepackage{amsmath, amssymb,mathtools}
\usepackage[usenames,dvipsnames]{xcolor}
\usepackage{graphicx}
\usepackage{subfigure}
\usepackage{minitoc}
\usepackage{tikz}
\usepackage{enumerate}
\usepackage[margin=0.96 in]{geometry}
\usepackage{bbm}
\usepackage[numbers,sort&compress]{natbib}
\usepackage[colorlinks=true]{hyperref}
\hypersetup{urlcolor=blue, citecolor=red}
\usepackage{crossreftools}

\DeclarePairedDelimiter{\abs}{\lvert}{\rvert}
\DeclarePairedDelimiter{\norm}{\lVert}{\rVert}
\DeclarePairedDelimiter{\set}{\{}{\}}

\DeclareMathAlphabet{\mathup}{OT1}{\familydefault}{m}{n}
\newcommand{\dx}[1]{\mathop{}\!\mathup{d} #1}

\DeclarePairedDelimiter{\prt}{(}{)}
\DeclarePairedDelimiter{\brk}{[}{]}

\newcommand{\R}{{\mathbb R}}
\newcommand{\Rd}{{\mathbb R^d}}
\newcommand{\curlyH}{\mathcal{H}}

\usepackage{amsthm}
\theoremstyle{plain}
\newtheorem{theorem}{Theorem}[section]
\newtheorem{lemma}[theorem]{Lemma}
\newtheorem{proposition}[theorem]{Proposition}

\theoremstyle{remark}
\newtheorem{remark}[theorem]{\bf Remark}
\newtheorem{definition}[theorem]{\bf Definition}

\newcommand{\ie}{\emph{i.e.}}
\newcommand{\cf}{\emph{cf.}\;}
\newcommand{\0}{_0}
\renewcommand{\i}{^{(i)}}
\newcommand{\1}{^{(1)}}
\newcommand{\2}{^{(2)}}
\newcommand{\epsnu}{_{\epsilon,\nu}}

\newcommand{\sign}{\mathrm{sign}}
\newcommand{\ds}{\displaystyle}
\newcommand{\ddt}{\frac{\dx{}}{\dx t}}
\newcommand{\partialt}[1]{\frac{\partial #1}{\partial t}}

\newcommand{\grad}{\nabla}
\renewcommand{\div}{\nabla\cdot}

\newcommand{\Lap}{\Delta}
\newcommand{\indicator}{\mathbbm{1}}

\begin{document}

\title[Inviscid Limit for a Two-Species Tumour Model]{A degenerate cross-diffusion system as the Inviscid Limit of a nonlocal tissue growth model}

\author{Noemi David$^{1}$} \author{Tomasz D\k{e}biec$^{2}$}
\author{Mainak Mandal$^{3}$}
\author{Markus Schmidtchen$^{3}$}

\address{$^{1}$ Institut Camille Jordan, Université Claude Bernard - Lyon 1, 21 Av. C. Bernard, 69100 Villeurbanne, France (ndavid@math.univ-lyon1.fr)}
\address{$^{2}$ Institute of Applied Mathematics and Mechanics, University of Warsaw, Banacha 2, 02-097 Warsaw, Poland (t.debiec@mimuw.edu.pl).}
\address{$^{3}$ Institute of Scientific Computing, Faculty of Mathematics, TU Dresden
	(mainak.mandal@tu-dresden.de,  markus.schmidtchen@tu-dresden.de).}

\maketitle
\begin{abstract}
    In recent years, there has been a spike in the interest in multi-phase tissue growth models. Depending on the type of tissue, the velocity is linked to the pressure through Stoke's law, Brinkman's law or Darcy's law. While each of these velocity-pressure relations has been studied in the literature, little emphasis has been placed on the fine relationship between them. In this paper, we want to address this dearth in the literature providing a rigorous argument that bridges the gap between a viscoelastic tumour model (of Brinkman type) and an inviscid tumour model (of Darcy type).\\[0.5em]
\end{abstract}{}

\vskip .4cm
\begin{flushleft}
    \noindent{\makebox[1in]\hrulefill}
\end{flushleft}
	 Parabolic-Hyperbolic Cross-Diffusion Systems;\\[-2.em]
\begin{flushright}
    \noindent{\makebox[1in]\hrulefill}
\end{flushright}
\vskip 1.5cm

\section{Introduction}
\subsection{A Beautiful Link between Local and Nonlocal Dispersal.}
In recent years, a beautiful link between \emph{nonlocal dispersal} and \emph{nonlinear, local dispersal} has been experiencing a renaissance. By nonlocal dispersal, we refer to dynamics governed by the equation
\begin{align}
    \label{eq:intro-nonloc}
    \partialt n = \nabla \cdot \prt*{n \nabla K_\epsilon\star n},
\end{align}
where the \emph{interaction kernel}, $K_\epsilon$, is a nonnegative, radial function with unit mass.
In the context of dispersal, it is important and commonly assumed that the kernel is, furthermore, strictly decreasing in the radial variable. 
Now, if the interaction kernel is scaled in such a way that the interactions become stronger and more localised, concurrently approaching a Dirac delta, \ie, $K_\epsilon \to \delta_0$, as $\epsilon\to0$, solutions to the nonlocal dispersion equation should, in a sense, converge to solutions of
\begin{align}
    \label{eq:intro-pme}
    \partialt n = \nabla \cdot \prt*{n \nabla \delta_0 \star n} = \nabla \cdot \prt*{n \nabla n}.
\end{align}
Of course, it is immediately clear that this limit is singular in the sense that the very nature of the equation changes from transport to degenerate parabolic. Both equations as well as the meta-problem of connecting both viewpoints have received considerable attention from different communities.

\subsection{Applications in Biomathematical Modelling.}
Nonlinear dispersal plays an important role in the evolution of large systems of interacting particles, \cite{NPT2010, CCY2019, BBPW2021}. 
The importance of nonlinear, local dispersal is particularly relevant in contexts of mathematical biology. While linear diffusion is used frequently in order to incorporate the random dispersal of particles, the tendency to use nonlinear diffusion for random dispersal represents a paradigm shift. It was ushered in and motivated by several independent biological experiments (see \cite{BT1983, Mor1950, Mor1954, Car1971}) that were suggesting that the diffusion coefficient depends on the local density, \ie, 
\begin{align*}
    \partialt n = \nabla \cdot(D(n) \nabla n),
\end{align*}
with $D(n) = n$, in contrast to linear diffusion, $D(n)=D$. Particularly in the context of tissue growth and multi-phase tumour modelling, a balance law reminiscent of a nonlinear Fisher-KPP equation, \ie,
\begin{align*}
    \partialt n + \nabla \cdot(n v) = n G(n),
\end{align*}
is used, where the velocity is related to the pressure gradient, $v = -\nabla p = -\nabla n$ and $G$ represents growth-death dynamics, most notably, \cf \cite{PQV14}, which acts as a starting point for this work.

\subsection{Aggregation Equation and Nonlocal Dispersal.}
We return to Eq.~\eqref{eq:intro-nonloc} which has received a lot of attention in its own right. It is customarily referred to as the \emph{aggregation equation}, \cf 
\cite{CCS2018, VS2015, ME1999, CDFLS2011}, and references therein. Using the method of matched asymptotic expansions, the authors of \cite{BCR17,  BBC2021} were able to systematically establish a link between Brownian particles with finite size, where the size-exclusion effects are incorporated by a very short-range, strong interaction potential, $W_\epsilon$, similarly to Eq.~\eqref{eq:intro-nonloc} and nonlinear diffusion similar to Eq.~\eqref{eq:intro-pme}. In a regime of low volume fraction and extremely localised repulsion, they showed that the probability density of a representative particle evolves according to a linear diffusion equation with a nonlinear correction of the form as in Eq.~\eqref{eq:intro-pme}. Indeed, systems of interacting particles with size-exclusion effects share one important feature, the presence of degenerate diffusion, as observed in various contexts, \cf \cite{BDPS2010, BC2012, BC2012a}.

\subsection{Analytical Approximations in the Literature.}
Concomitantly, the link between nonlocal dispersal and local nonlinear diffusion has also been addressed rigorously. Indeed, to the best of our knowledge, it dates back to a deterministic particle method for nonlinear diffusion equations introduced in \cite{DM1990}. An idea of the proof based on a double-mollification argument theory was given in \cite{LM2001}. Subsequently, fully rigorous proofs were given in \cite{CCP2019, BE2022}, for the passage from Eq.~\eqref{eq:intro-nonloc} to Eq.~\eqref{eq:intro-pme}, culminating in the very recent work by Carrillo, Esposito, and Wu \cite{CEW2023} that provides a full answer to nonlocal approximations of general degenerate diffusion equations. A remarkable fact about all of the aforementioned results is their technique which is based on an astute mollification argument. In its essence, it allows for a chain rule in the nonlocal equation mimicking the entropy dissipation structure of the local limit equation. 

Since the methods are based on techniques from optimal transportation they are unable to incorporate growth dynamics and the extension to multiple species ($i=1,2$) with joint population pressure, \ie, 
\begin{align}
    \label{eq:intro-multi-pme}
    \partialt n\i = \nabla \cdot (n\i \nabla p(n\1 + n\2)),
\end{align}
coupled through constitutive pressure law, $p$, is anything but straightforward. Moreover, any attempts to show existence of solutions (even in one dimension, leave alone multiple dimensions) using a nonlocal approximation have remained futile so far. We will address this gap in the literature in this work.

\subsection{Cross-Interaction Systems.}
Unlike other multi-species cross-diffusion systems that incorporate size exclusion effects that were recently introduced and considered, see \cite{ABFS2022, DEF2018, BDPS2010, Jun2015, BBP2017, ARSW2020, BCS2022}, System~\eqref{eq:intro-multi-pme} assumes a rather prominent place. This is due to the high symmetry in the velocity-pressure relation which renders the system non-strictly parabolic and is only convex rather than strictly convex. If convexity is completely lost, convexification arguments can be employed \cite{DSY2023}, however, the approximation might not converge to solutions of the original system. Returning to System~\eqref{eq:intro-multi-pme}, from the very early works, it has been reported, shown numerically, and proven analytically, that solutions may exhibit a propagation of segregation property. Thereupon, even arbitrarily smooth initial data may lose regularity at the onset of sharp interfaces between both densities. Formally speaking, this sharp interface is transported with the same velocity, $v=-\nabla p$, rendering the quest of obtaining regularity of the individual species, $n\i$,  a demanding challenge. Indeed, any regularity better than bounded variation cannot be expected, \cite{BHIMW2015, CFSS2017, BPPS2019}. A different approach to prove the existence of solutions relies on the strong compactness of the pressure gradient, either by variations of the Aronson-B\'enilan estimate \cite{BPPS2019, GPS2019, Jac2022} or convergence of the norm \cite{PX2020, LX2021, Dav2022, Jac2021}.

Historically, System \eqref{eq:intro-multi-pme}
first appeared in \cite{BT1983} in the context of epidemiological modelling of polymorphic populations and in \cite{GP1984} with applications to population dynamics that avoid overcrowding. It was subsequently analysed in \cite{BGHP1985, BGH1987, BGH1987a}. 
These early works were ensued by \cite{BDPM2010, BHIM2012} which study a variation of System \eqref{eq:intro-multi-pme} incorporating growth dynamics first in one and then in arbitrary spatial dimensions for tissues with cell-cell contact inhibitions. While the argument relied upon the introduction of transformed variables, a different regularisation approach was introduced in \cite{GS2014}. The models with growth and joint population pressure were applied and further scrutinised in \cite{MT2015, CMSTT2019} for human embryonic kidney cells and they are able to explain the phase separation of normal and abnormal tissue.
Following a diametrically different approach based on a splitting scheme, the authors of \cite{CFSS2017} were able to remove restrictive assumptions on the initial data, namely those of being bounded away from zero or strictly ordered in one dimension. Successively, assumptions on initial data and space dimension were removed, see \cite{GPS2019, BPPS2019, PX2020, LX2021, Dav2022}. 
Alongside these analytical advances and the growing understanding of the system, it is also worthwhile mentioning several results on travelling wave solutions, steady states, and pattern formation for variations of Eq. \eqref{eq:intro-multi-pme}, see \cite{BHIMW2015, CHS2017, BHIHMW2020, BCPS2020, GG2020}.  
To this day, the existence of solutions to Eq.~\eqref{eq:intro-multi-pme} with different diffusivity, apart from \cite{GP1984}, is widely open. Similarly, the inclusion of different drifts remains a challenging open question with two very partial results given in \cite{KM2018, CHS2017}, in one dimension and only for specific drifts or strong assumptions on the drifts. 

The interest in System \eqref{eq:intro-multi-pme} and the wealth of new methods for proving existence leads us to believe that our work connects pleasingly with the advances in the literature in the last few years: not only do we prove the existence of solutions, but we propose yet a new approach with the humble hope it may interest the community.

\subsection{Goal and Structure of this Paper.}

The starting point of our analysis is the model
\begin{align}
	\label{eq:Brinkman}
	\left\{
	\begin{array}{rl}
		\ds \partialt{n_\nu\i} - \div\prt*{n_\nu\i \grad W_\nu} \!\!\!& = \ds n_\nu\i G\i(n_\nu), \\[1em]
		\ds -\nu \Lap W_\nu +W_\nu \!\!\! & = \ds  n_\nu,
	\end{array}
	\right.
\end{align}
where $n\i = n\i(x,t)$, $i=1,2$, represents the normal (resp.\ abnormal) cell density  at location $x\in\R^d$ at time $t \in [0,T)$. Here, $n_\nu = n_\nu\1 + n_\nu\2$ denotes the total population, proposed in \cite{DS2020, DPSV2021} as a two-species version of \cite{PV2015}. The system is equipped with some non-negative initial data $n^{{(i)},\mathrm{in}} \in L^1(\R^d)\cap L^\infty(\R^d)$.
Here, $\nu>0$ is a fixed positive constant  representing a measure of the tissue's viscosity. The elliptic equation linking the macroscopic velocity potential, $W_\nu$, with the density, $n_\nu$, is typically referred to as \emph{Brinkman's law}, see for instance \cite{All91, PV2015}. The growth of the two densities is assumed to be modulated by density-dependent growth rates, $G^{(i)}$, \cite{BD09, RBEJPJ10}. Throughout, we shall make the following assumptions on the growth rate:\\
\begin{itemize}
	\setlength\itemsep{1em}
	\item[(G1)] regularity: $G^{(i)} \in C^{1} (\mathbb{R})$, for $ i = 1,2 $.
	\item[(G2)] monotonicity: $\partial_{n} G^{(i)}\leq - \alpha < 0$, for some $
			\alpha > 0$, for $ i = 1,2 $.
	\item[(G3)] critical densities: there exists $ \bar n>0$ such that $G^{(i)}(n) \leq 0, \ \forall n \geq \bar n$, and $ i = 1,2 $.\\[0.25em]
\end{itemize}
Formally, it is easy to observe that solutions to the Brinkman model, System \eqref{eq:Brinkman}, converge to solutions to the Darcy model
\begin{equation}
	\begin{aligned}
		\label{eq:Darcy}
		\ds \partialt{n_0\i} - \div\prt*{n_0\i \grad n_0} & = \ds n_0\i G\i(n_0), \quad i=1,2,
	\end{aligned}
\end{equation}
where $n_0\i = \lim_{\nu\to0} n_\nu\i$, and $n\0 = n\0\1+n\0\2$. Of course, this argument only holds on a formal level.  The challenge is immediately clear when directing our focus on the nature of the two models -- the Brinkman model is of transport type while the Darcy model is a degenerate parabolic equation. In a sense, solutions to the limit equation, System \eqref{eq:Darcy}, have higher regularity than solutions for $\nu>0$, and this gain of regularity at the limit classifies this question as a singular limit problem. 

It is the goal of this paper to establish the limit process rigorously. Before we state the precise statement of the main result let us briefly introduce the notion of weak solutions to the two equations.

\begin{definition}[Weak Solutions - Brinkman]
    \label{def:wk_sol_brinkman}
	We say the pair  $(n_\nu\1, n_\nu\2) \geq 0$ is a weak solution to System \eqref{eq:Brinkman}, if $n\i \in L^\infty(0,T; L^1(\Rd) \cap L^\infty(\Rd))$, for $i=1,2$, and there holds
	\begin{align*}
		\int_0^T\!\!\!  \int_\Rd &n_\nu\i \partialt \varphi \dx x \dx t 
        - \int_0^T \!\!\! \int_\Rd n_\nu\i \nabla W_\nu \cdot \nabla \varphi \dx x \dx t
		\\
        &= -\int_0^T \!\!\! \int_\Rd \varphi n_\nu\i G\i(n_\nu ) \dx x \dx t -\int_\Rd \varphi(x,0)n^{(i), \mathrm{in}}(x) \dx x ,
	\end{align*}
	for any test function $\varphi \in C_{c}^{\infty}(\Rd \times [0,T))$, as well as 
	\begin{align*}
		- \nu \Lap W_\nu + W_\nu = n_\nu,
	\end{align*}
    almost everywhere in $\Rd\times(0,T)$.
\end{definition}
Similarly, let us introduce our notion of weak solutions to the limiting equation.
\begin{definition}[Weak Solutions - Darcy]
    \label{def:wk_sol_darcy}
	We say a pair of nonnegative functions,  $(n\0\1, n\0\2)$ is a weak solution to System \eqref{eq:Darcy}, if $n_0\i \in L^\infty(0,T; L^1(\Rd)\cap L^\infty(\Rd))$ and $ n_0 \in L^2(0,T; H^1(\Rd))$ and it satisfies
	\begin{align*}
		\int_0^T \!\!\! \int_\Rd &n\0\i \partialt \varphi \dx x \dx t 
        - \int_0^T \!\!\! \int_\Rd n\0\i \nabla n\0 \cdot \nabla \varphi \dx x \dx t\\
        &=- \int_0^T \!\!\! \int_\Rd \varphi n\0\i G\i(n\0 ) \dx x \dx t - \int_\Rd \varphi(x,0) n\0^{(i), \mathrm{in}}(x) \dx x,
	\end{align*}
	for any test function $\varphi \in C_{c}^{\infty}(\Rd\times[0,T))$.
\end{definition}

In our analysis of the inviscid limit, we shall make use of the entropy structure of the primitive system and the limit systems, respectively. Let us introduce the usual entropy functional
\begin{equation*}
        \curlyH[f](t) := \int_\Rd f(x, t)(\log f(x, t) - 1) \dx x.
\end{equation*}
Finally, let us recall that any solution of Brinkman's equation can be expressed as convolution with the fundamental solution, denoted by $K_\nu$, \ie,  the solution of 
\begin{align*}
    - \nu \Delta W_\nu + W_\nu = n_\nu,
\end{align*}
can be expressed as
\begin{align*}
    W_\nu = K_\nu \star n_\nu.
\end{align*}
With these definitions at hand, we are now able to state the main result of this paper.

\begin{theorem}[Inviscid Limit]
\label{thm:inviscid_limit}
    Let $(n_\nu\i)_{\nu>0}$ be a sequence of solutions to System \eqref{eq:Brinkman} in the sense of Definition \ref{def:wk_sol_brinkman}. Suppose that the initial data $n_\nu^{\mathrm{in}}:=n_\nu^{(1),\mathrm{in}}+n_\nu^{(2),\mathrm{in}}$ satisfies
    \begin{equation*}
        \sup_{\nu>0}\mathcal{H}[n_\nu^{\mathrm{in}}] < \infty,\quad \sup_{\nu>0}\int_{\Rd}|x|^2n_\nu^{\mathrm{in}}\dx x < \infty,\quad \sup_{\nu>0}\int_{\Rd} n_\nu^{\mathrm{in}} K\star n_\nu^{\mathrm{in}}\dx x < \infty.
    \end{equation*}
    Then, there exists a subsequence, still denoted $(n_\nu\i)_\nu$ and functions $n_0\i\in~L^\infty(0,T; L^1(\Rd)\cap L^\infty(\Rd))$, such that
    \begin{align}
        n_\nu\i \rightharpoonup n\0\i,
    \end{align}
    weakly in $L^2(0,T; L^2(\Rd))$, as well as
    \begin{align}
        n_\nu \to n\0, 
    \end{align}
    with $n_0:= n_0\1 + n_0\2$,
    strongly in $L^2(0,T; L^2(\Rd))$, and
     \begin{align}
        W_\nu \to n\0, 
    \end{align}
    strongly in $L^2(0,T; H^1(\Rd))$. Furthermore, $(n_0\1, n_0\2)$ is a weak solution to System \eqref{eq:Darcy} in the sense of Definition \ref{def:wk_sol_darcy}.
\end{theorem}

\section{A Priori Estimates}
We begin by summarising some fundamental regularity results for System \eqref{eq:Brinkman} that were obtained recently in a series of works, see \cite{PV2015, DS2020, DPSV2021}. Most notably, for any $\nu>0$, there is a unique solution $(n_\nu\1, n_\nu\2)$ to System \eqref{eq:Brinkman} satisfying:\\
\begin{itemize}
    \setlength\itemsep{1em}
	\item[(B1)] $n_\nu\i \in L^\infty(0,T; L^1(\Rd))$,  for $i=1,2$,
	\item[(B2)] $0 \leq n_\nu\i(x,t) \leq \bar n$, almost everywhere, with $i=1,2$,
	\item[(B3)] $W_\nu\in L^\infty(0,T; L^1(\Rd) \cap L^\infty(\Rd))$,\\[0.25em]
 \end{itemize}
uniformly in $\nu>0$. Moreover, the potential, $W_\nu$, enjoys the following regularity properties:\\
 \begin{itemize}
    \setlength\itemsep{1em}
    \item[({\crtcrossreflabel{B4}[hyp:B4]})] $W_\nu\in L^\infty(0,T; W^{1,q}(\Rd))$, for $1\leq q\leq \infty$,
    \item[({\crtcrossreflabel{B5}[hyp:B5]})] $D^2 W_\nu\in L^\infty(0,T; L^q(\Rd)$, for $1< q <\infty$,\\[0.25em]
\end{itemize}
for any given $\nu>0$. However, let us stress that the last two statements are not claimed to hold uniformly in $\nu$. Indeed, it is a substantial part of this work to establish~(\ref{hyp:B4}) uniformly in $\nu$, at least in $L^2(0,T; H^1(\Rd))$. Moreover, let us point out that~(\ref{hyp:B5}) cannot be expected to hold true as, in the limit, $\Delta n_0$ is only a measure in general.

\subsection{Entropy Dissipation Inequality.}
This section is dedicated to establishing an entropy dissipation inequality for the joint population density. Indeed, for a solution, $(n_\nu\1, n_\nu\2)$, to System \eqref{eq:Brinkman}, we can show that $n_\nu = n_\nu\1 + n_\nu\2$ satisfies
\begin{align*}
    \curlyH[n_\nu](T) - \curlyH[n_\nu^{\mathrm{in}}]
    & - \int_0^T \!\!\! \int_\Rd  n_\nu \Lap W_\nu  \dx x \dx t\\
    & \leq \int_0^T \!\!\! \int_\Rd \log n_\nu \left[ n^{(1)}_{\nu}G^{(1)}(n_\nu) +n^{(2)}_{\nu}G^{(2)}(n_\nu) \right] \dx x \dx t.    
\end{align*}
In order to establish this inequality rigorously, let us begin by regularising the equations satisfied by the two densities, $n\i_\nu$. Specifically, for any $\epsilon>0$, we consider
\begin{align}
    \label{eq:regularised-equation_i}
	\left\{
	\begin{array}{rll}
		\ds \partialt {n\i\epsnu} - \epsilon \Lap n\i\epsnu \!\!\! & = \nabla \cdot
		(n\i\epsnu \nabla W_\nu) + n\i\epsnu  G\i(n_\nu)  & \text{in } \Rd\times(0,T), \\[0.75em]
		n\i\epsnu(x,0) \!\!\!        & =
		n\epsnu^{(i),\mathrm{in}}(x), &
	\end{array}
	\right.
\end{align}
for $i=1,2$, where $n_\nu$ and $W_\nu$ are the unique solutions to System~\eqref{eq:Brinkman}, and $n\epsnu^\mathrm{(i),in} \in C_c^\infty(\Rd)$ denotes some nonnegative regularised initial data satisfying $n\epsnu^\mathrm{(i),in} \to n_\nu^\mathrm{(i),in}$ in $L^1(\Rd)$ and almost everywhere in $\Rd$.

The existence of a pair of unique $L^2(0,T;H^1(\Rd))$-solutions $(n\1\epsnu,n\2\epsnu)$ is folklore and\\
\begin{itemize}
    \setlength\itemsep{1em}
    \item [(i)] $ n\i\epsnu \in L^{\infty}(0,T;L^1(\Rd)\cap L^\infty(\Rd)) \cap L^{2}(0,T;H^{1}(\Rd))$,
    \item [(ii)] $\partial_{t} n\i\epsnu \in  L^2(0,T;L^{2}(\Rd))$.\\
\end{itemize}
We will now discuss some important properties of the total regularised density $n\epsnu := n\1\epsnu+n\2\epsnu$, which satisfies the equation
\begin{align}
    \label{eq:regularised-equation}
	\left\{
	\begin{array}{rll}
		\ds \partialt {n\epsnu} - \epsilon \Lap n\epsnu \!\!\! & = \nabla \cdot
		(n\epsnu \nabla W_\nu) + n\1\epsnu  G\1(n_\nu) +  n\2\epsnu G\2(n_\nu),            & \text{in } \Rd\times(0,T), \\[0.75em]
		n\epsnu(x,0) \!\!\!        & =
		n\epsnu^{\mathrm{in}}(x):=n\epsnu^{(1),\mathrm{in}}(x)+n\epsnu^{(2),\mathrm{in}}(x). &
	\end{array}
	\right.
\end{align}

\begin{lemma}
    \label{lem:second_moment}
	Let $n\epsnu$ be the solution to Eq.~\eqref{eq:regularised-equation}. Assume that there is a constant, $C>0$, such that 
\begin{align}
    \sup_{\epsilon >0} \int_\Rd n\epsnu^{\mathrm{in}}(x) |x|^2 \dx x \leq C,
\end{align}
\ie, the initial data has finite second moment. Then, the second moment remains bounded and there holds
\begin{align*}
	    \sup_{t\in [0,T]}\int_\Rd n\epsnu(x,t) \abs{x}^{2}  \dx x\leq C + C \int_0^T\!\!\!\int_\Rd n\epsnu |\nabla W_\nu|^2 \dx x \dx t =:C(\nu),
\end{align*}
for some constant $C(\nu)>0$, independent of $\epsilon$.
\end{lemma}
\begin{proof}
	We compute
	\begin{align*}
		\frac{1}{2}  \ddt \int_\Rd n\epsnu \abs{x}^{2} \dx x
		 & =  \epsilon d \int_\Rd n\epsnu \dx x + \int_\Rd n\epsnu x \cdot \nabla
		W_\nu \dx x \\
        &\qquad +  \frac12 \int_\Rd |x|^2 \prt*{
		n\1\epsnu G\1(n_\nu) + n\2\epsnu G\2(n_\nu)} \dx x      \\[0.3em]
		& \leq \epsilon d \norm{n\epsnu}_{L^\infty(0,T; L^1(\Rd))} +
		\prt*{\frac12 + \max_{i=1,2}\norm{G^{(i)}}_{L^\infty(0,\bar n)}}\int_\Rd n\epsnu
		|x|^2\dx x \\
		 & \qquad + \frac12 \int_\Rd n\epsnu |\nabla W_\nu|^2 \dx x\\[0.3em]
		& \leq C_1 + \frac12\int_\Rd n\epsnu |\nabla W_\nu|^2 \dx x + \frac{C_2}{2} \int_\Rd  n\epsnu \abs{x}^{2} \dx x.
	\end{align*}
	Then, by Gronwall's lemma, we get
	\begin{align*}
        \frac12 &\int_\Rd \abs{x}^{2} n\epsnu (x,t) \dx x\\
        &\leq 
        \brk*{\frac12
        \int_\Rd \abs{x}^{2} n\epsnu^{\text{in}}(x) \dx x + C_1 T +
        \frac12 \int_0^T \!\!\! \int_\Rd  n\epsnu \abs{\grad W_\nu}^{2}
        \dx x\dx t
        }  \exp{(C_{2}T)}.
	\end{align*}
This completes the proof.
\end{proof}
Let us point out that the above lemma implies also uniform second moment control for the individual species, $n\i\epsnu$, $i=1,2$.

\begin{lemma}[Entropy Bounds]
    \label{lem:entropy-bounds}
    Let $n\epsnu$ be the solution to Eq. \eqref{eq:regularised-equation}. Then there holds
    \begin{align*}
        \sup_{\epsilon > 0} \sup_{t\in[0,T]} \int_\Rd n\epsnu (t) |\log n\epsnu (t)| \dx x \leq C(\nu),
    \end{align*}
    for some constant, $C(\nu)>0$, independent of $\epsilon$.
\end{lemma}

\begin{proof}
Let us consider
\begin{align*}
    \int_\Rd n\epsnu \abs{\log n\epsnu} \dx x 
    &= \int_{\set{n\epsnu \geq 1}} n\epsnu \log n\epsnu \dx x - \int_{\set{n\epsnu < 1}} n\epsnu \log n\epsnu \dx x\\[0.5em]
    &\leq \norm{n\epsnu}_{L^\infty(0,T; L^\infty(\Rd))} \norm{n\epsnu}_{L^\infty(0,T;L^1(\Rd))}  + J,
\end{align*}
where 
\begin{align*}
    J := - \int_{\set{n\epsnu < 1}} n\epsnu \log n\epsnu \dx x.
\end{align*}
In order to estimate $J$, let $N(x)$ denote the standard normal Gaussian. Then, we have 
\begin{align*}
    J &= - \int_{\set{n\epsnu < 1}} n\epsnu \log n\epsnu \dx x \\
    &= - \int_\Rd \frac{n\epsnu \indicator_{\set{n\epsnu < 1}}}{N} \log\prt*{\frac{n\epsnu \indicator_{\set{n\epsnu < 1}} }{N}} N \dx x  + \int_\Rd n\epsnu \indicator_{\set{n\epsnu < 1}} |x|^2 \dx x\\
    &\leq - \int_\Rd \frac{n\epsnu \indicator_{\set{n\epsnu < 1}}}{N} \log\prt*{\frac{n\epsnu \indicator_{\set{n\epsnu < 1}} }{N}} N \dx x + C(\nu),
\end{align*}
having used the second-order moment bound from Lemma \ref{lem:second_moment}. Applying Jensen's inequality to the first term, we observe
\begin{align*}
    - \int_\Rd & \frac{n\epsnu \indicator_{\set{n\epsnu < 1}}}{N} \log\prt*{\frac{n\epsnu \indicator_{\set{n\epsnu < 1}} }{N}} N \dx x\\
    &\leq -\prt*{\int_\Rd \frac{n\epsnu \indicator_{\set{n\epsnu < 1}} }{N} N \dx x} \log \prt*{\int_\Rd \frac{n\epsnu \indicator_{\set{n\epsnu < 1}} }{N} N \dx x}\\
    & \leq e^{-1}.
\end{align*}
In conclusion, we obtain $J \leq C(\nu) + e^{-1}$, whence
\begin{align*}
    \int_\Rd n\epsnu |\log n\epsnu| \dx x \leq \norm{n\epsnu}_{L^\infty(0,T; L^\infty(\Rd))} \norm{n\epsnu}_{L^\infty(0,T;L^1(\Rd))}  + C(\nu) + e^{-1},
\end{align*}
which concludes the proof.
\end{proof}

\begin{proposition}
    \label{prop:regular_n_bounds}
	Let $(n\epsnu)_{\epsilon>0}$ be a family of solutions to Eq. \eqref{eq:regularised-equation}. Then, there holds
\begin{align*}
    \sqrt{\epsilon} \|\nabla n\epsnu\|_{L^2(0,T; L^2(\Rd))} \leq C,
\end{align*}
for some constant $C>0$ independent of $\epsilon$.
\end{proposition}
\begin{proof}
	Testing Eq. \eqref{eq:regularised-equation} by $n_{\epsilon,\nu}$ and integrating by parts, we obtain
	\begin{align*}
        \frac12 \dfrac{\dx }{\dx t} &\int_\Rd
        n\epsnu^2  \dx x + \epsilon \int_\Rd \abs{\grad
        n\epsnu}^{2} \dx x\\
        & =  \frac12 \int_\Rd 
        n\epsnu^2 \Lap W_\nu \dx x + \int_{\Rd}^{} n\epsnu
        \prt*{n\1\epsnu G^{(1)}(n_\nu) + n\2\epsnu G^{(2)}(n_\nu)} \dx x\\
        &\leq C\int_{\Rd} n\epsnu^2\dx{x},
	\end{align*}
	using~(\ref{hyp:B5}). An application of Gronwall's lemma yields the statement.
\end{proof}
\begin{proposition}[Compactness of Solutions to the Regularised Equation]
    \label{prop:n_epsnu_compact}
	The families $ (n\1\epsnu)_{\epsilon > 0}, (n\2\epsnu)_{\epsilon > 0},(n\epsnu)_{\epsilon > 0}  $ are compact in $L^1(0,T;L^{1}(\Rd))$.
\end{proposition}
\begin{proof}
    We will prove compactness for the individual densities, $n\i\epsnu$.
    It is easy to see that $\partial_t n\i\epsnu$ is bounded uniformly (with respect to $\epsilon$) in $L^2(0,T; H^{-1}(\Rd))$. For space compactness, we follow the strategy of Belgacem-Jabin~\cite[Proposition~3.1]{BelgacemJabin}. They guarantee local compactness by considering the evolution of the following quantity
    \begin{equation*}
        Q_\epsilon(t) := \iint_{\mathbb{R}^{2d}} K_h(x-y)|n\i\epsnu(x,t)-n\i\epsnu(y,t)|\dx x\dx y,
    \end{equation*}
    where $(K_h)_{0<h<1}$ is a family of nonnegative, smooth functions such that
    \begin{equation*}
        \norm{K_h}_{L^1(\Rd)} \sim \abs{\log h},\quad \text{and}\quad |x||\nabla K_h(x)| \leq C K_h(x).
    \end{equation*}
    Then, the sequence $(n\i\epsnu)_{\epsilon>0}$ is compact in $L^1_{\mathrm{loc}}(\Rd)$, provided that
    \begin{equation*}
        \lim_{h\to 0}\; \limsup_{\epsilon\to0}\; |\log h|^{-1} Q_\epsilon(t) = 0.
    \end{equation*}
    For the precise statement see~\cite[Lemma~3.1]{BelgacemJabin}.
    Our situation is of course simpler since Eq.~\eqref{eq:regularised-equation_i} is linear and the velocity field is independent of $\epsilon$. The only term that we need to be careful about is the reaction. The contribution of this term in $\ddt Q_\epsilon$ is
    \begin{align*}
        \iint_{\mathbb{R}^{2d}} & K_h(x-y)  
        \sign\big(n\i\epsnu(x)-n\i\epsnu(y)\big)\big[n\i\epsnu(x)G\i(n_\nu(x))-n\i\epsnu(y)G\i(n_\nu(y))\big] \dx x\dx y\\
        & \leq \norm{n\i\epsnu}_{L^\infty(0,T;L^\infty(\Rd))}\iint_{\mathbb{R}^{2d}} K_h(x-y) \abs*{G\i(n_\nu(x))-G\i(n_\nu(y))} \dx x\dx y \\
        &\qquad +\norm{G\i}_{L^\infty(\R)}\iint_{\mathbb{R}^{2d}} K_h(x-y)\abs*{n\i\epsnu(x)-n\i\epsnu(y)}\dx x\dx y\\
        &\leq  C\iint_{\mathbb{R}^{2d}} K_h(x-y) \abs*{n_\nu(x)-n_\nu(y)} \dx x\dx y + CQ_\epsilon(t),
    \end{align*}
    having used the fact that $G\i$ is Lipschitz. Since $n_\nu\in L^\infty(0,T;L^1(\Rd))$, the first term on the right-hand side satisfies
    \begin{equation*}
        \lim_{h\to0}|\log h|^{-1}\iint_{\mathbb{R}^{2d}} K_h(x-y)|n_\nu(x,t)-n_\nu(y,t)|\dx x\dx y = 0.
    \end{equation*}
    We conclude that $n\i\epsnu$ is compact in $L^1(0,T;L^1_{\mathrm{loc}}(\Rd))$, for $i=1,2$. Finally, using the second moment control of Proposition~\eqref{lem:second_moment}, we can deduce global compactness.
    
\end{proof}
By the compactness of solutions to the regularised equations, we can always extract subsequences, still denoted $(n\i\epsnu)_{\epsilon > 0}$, such that $ n\i\epsnu \to n\i_\nu$ 
almost everywhere, as $ \epsilon \to 0$. Using Proposition \ref{prop:regular_n_bounds}, it is easy to identify the limit of $n\epsnu$ as the unique solution satisfied by the sum of the $n_\nu\i$, \cf System~\eqref{eq:Brinkman}.

Passing to the limit $\epsilon\to0$, we can now ensure that the limit $n_\nu$ satisfies the second moment bound, and hence the entropy bounds, uniformly in $\nu$.

\begin{proposition}
\label{lem:second_moment_nu}
    The limit $n_\nu$ of the sequence $(n\epsnu)_{\epsilon>0}$ satisfies
    \begin{equation*}
        \sup_{\nu>0}\sup_{t\in[0,T]}\int_{\Rd} n_\nu|x|^2\dx x < \infty.
    \end{equation*}
\end{proposition}
\begin{proof}
    From the proofs of Lemmata~\ref{lem:second_moment} and~\ref{lem:entropy-bounds} we see that the only obstruction to $\nu$-uniform bounds is caused by the quantity $\iint n\epsnu |\nabla W_\nu|^2$. We will now show that in the vanishing viscosity limit, this quantity is independent of $\nu$.
    Indeed, testing Eq.~\eqref{eq:regularised-equation} by $W_\nu$ and passing to the limit $\epsilon\to0$, we obtain
    \begin{equation*}
    \int_0^T\!\!\!\int_{\Rd} n_\nu|\nabla W_\nu|^2\dx x \dx t \leq C + \frac12\int_{\R^d}n_\nu^{\mathrm{in}}\prt*{K\star n_\nu^{\mathrm{in}}}\dx x \leq C.
    \end{equation*}
\end{proof}

We now state and prove the main result of this section, namely that the solutions to the equation for the total density, $n_\nu$, are entropy weak solutions.
\begin{proposition}[Entropy Inequality]\label{lem:entropy_inequality}
    The limit $n_\nu$ of the sequence $(n\epsnu)_{\epsilon>0}$ satisfies the entropy inequality
    \begin{equation}
    \label{eq:entropy_inequality}
        \begin{split}
    \mathcal{H}[n_\nu](T) - \mathcal{H}[n_\nu^{\mathrm{in}}]
			 & - \int_0^T \!\!\! \int_\Rd  n_\nu \Lap W_\nu \dx x \dx t\\
			&\leq\int_0^T \!\!\! \int_\Rd \log n_\nu\left[
			n^{(1)}_{\nu}G^{(1)}(n_\nu) +n^{(2)}_{\nu}G^{(2)}(n_\nu) \right] \dx x \dx t.
        \end{split}
    \end{equation}
\end{proposition}
\begin{proof}
Let $\delta > 0$ and $\phi \in C^{\infty}_{c} (\R^{d})$,
$\phi \geq 0 $. We consider the following regularised form of the entropy functional
	\begin{equation*}
		\begin{aligned}      
  \mathcal{H}_{\delta}^{\phi}[n_{\epsilon,\nu}](t) \coloneqq \int_{\R^d}^{}
			\left( n_{\epsilon,\nu}(x,t) + \delta \right)(\log(n\epsnu(x,t)+\delta)-1)
			\phi(x)\dx x,
		\end{aligned}
	\end{equation*}
 and by $\mathcal{H}^{\phi}[f]$ we denote the above functional with $\delta=0$.
 Given the regularity of $n\epsnu$ provided by parabolic theory, the weak form of Eq.~\eqref{eq:regularised-equation} can be written as
	\begin{equation*}
		\begin{aligned}
        \int_0^T \!\!\! \int_\Rd  \partialt{n\epsnu}
			\varphi(x,t)  &+ n\epsnu \grad \varphi(x,t) \cdot \grad W_\nu \dx x
			\dx t
			 + \epsilon \int_0^T \!\!\! \int_\Rd \grad n\epsnu
			\cdot \grad \varphi(x,t)\dx x \dx t \\
			 & = \int_0^T \!\!\! \int_\Rd \prt*{n\1\epsnu G\1(n_\nu) + n\2\epsnu G\2(n_\nu)}\varphi(x,t)\dx x \dx t,
		\end{aligned}
	\end{equation*}
    for any $\varphi \in L^2(0,T;H^1(\R^d))$. Choosing
	\begin{equation*}
		\begin{aligned}
			\varphi(x,t) \coloneqq \log(n\epsnu + \delta) \phi(x),
		\end{aligned}
	\end{equation*}
	we obtain
	\begin{equation}\label{eq:entropy_estimate_split}
		\begin{aligned}
			\int_0^T \dfrac{\mathup{d}{}}{\mathup{dt}} \mathcal{H}^{\phi}_{\delta} [n\epsnu](t)
			\dx t+ I^{\epsilon, \delta}_{1} + I^{\epsilon, \delta}_{2} +
			I^{\epsilon, \delta}_{3} + I^{\epsilon, \delta}_{4} = I^{\epsilon,
					\delta}_{5},
		\end{aligned}
	\end{equation}
	where
	\begin{equation*}
		\begin{aligned}
			I^{\epsilon, \delta}_{1} & =\int_0^T \!\!\!\int_{\R^{d}}^{} \frac{n\epsnu}{n\epsnu+
			\delta} \grad n\epsnu \cdot \grad  W_\nu \phi \dx x\dx t,                  \\
			I^{\epsilon, \delta}_{2} & = \int_0^T \!\!\! \int_\Rd n\epsnu \log(n\epsnu+\delta)\grad
			\phi \cdot \grad W_\nu \dx x\dx t,                                 \\
			I^{\epsilon, \delta}_{3} & = \epsilon \int_0^T \!\!\! \int_\Rd \frac{\abs{\grad
			n\epsnu}^{2}}{n\epsnu+\delta} \phi \dx x \dx t,                              \\
			I^{\epsilon, \delta}_{4} & = \epsilon \int_0^T \!\!\! \int_\Rd \log(n\epsnu
			+\delta)\grad \phi \cdot \grad n\epsnu \dx x \dx t,                          \\
			I^{\epsilon, \delta}_{5} & = \int_0^T \!\!\! \int_\Rd \log(n\epsnu+\delta) \prt*{n\1\epsnu G^{(1)}(n_\nu) + n\2\epsnu G^{(2)}(n_\nu)} \phi \dx x \dx t.
		\end{aligned}
	\end{equation*}

	Now we treat the terms individually. Starting with $ I_{1}^{\epsilon,\delta}$ we see
	\begin{equation*}
		\begin{aligned}
			I_{1}^{\epsilon,\delta}
			 & = \int_0^T \!\!\!\int_{\R^{d}}^{} \frac{n\epsnu}{n\epsnu+
			\delta} \grad n\epsnu \cdot \grad  W_\nu \phi \dx x\dx t,\\
			 & = \int_0^T \int_{\R^{d}}^{} \grad n\epsnu \cdot \grad W_\nu \phi
			\dx x\dx t-  \delta \int_0^T \int_{\R^{d}}^{} \frac{\grad n\epsnu \cdot
			\grad  W_\nu}{n\epsnu+ \delta}  \phi \dx x\dx t, \\
			 & = -\int_0^T \!\!\! \int_\Rd  n\epsnu \Lap W_\nu \phi
			\dx x\dx t- \int_0^T\int_\Rd n\epsnu\nabla W_\nu\cdot\nabla\phi \dx x \dx t \\
            &\qquad + \delta \int_0^T \int_{\R^{d}}^{}  \log(n\epsnu
			+ \delta) \Lap  W_\nu \phi \dx x\dx t +\delta \int_0^T \!\!\! \int_\Rd 
			\log(n\epsnu + \delta) \grad
			W_\nu\cdot \grad \phi \dx x \dx t.
		\end{aligned}
	\end{equation*}
    Passing to the limit $\epsilon\to 0$, we readily obtain
    	\begin{equation*}
		\begin{aligned}
			I_1^{\epsilon,\delta}\to I_{1}^{\delta}
			 & = -\int_0^T \!\!\! \int_\Rd  n_\nu \Lap W_\nu \phi
			\dx x\dx t- \int_0^T\int_\Rd n_\nu\nabla W_\nu\cdot\nabla\phi \dx x \dx t \\
            &\qquad + \delta \int_0^T \int_{\R^{d}}^{}  \log(n_\nu
			+ \delta) \Lap  W_\nu \phi \dx x\dx t                                + \delta \int_0^T \!\!\! \int_\Rd 
			\log(n_\nu + \delta) \grad
			W_\nu\cdot \grad \phi \dx x \dx t.
		\end{aligned}
	\end{equation*}
    For the last two integrals, we observe that
    \begin{align*}
        \delta \int_0^T \int_{\R^{d}}^{}  \log(n_\nu
		+ \delta) \Lap  W_\nu \phi \dx x\dx t                         + \delta \int_0^T \!\!\! \int_\Rd 
		\abs{\log(n_\nu + \delta)}^{} |\grad
		W_\nu\cdot \grad \phi| \dx x \dx t\leq C \delta\abs{\log\delta}. 
    \end{align*}
    Thus, when $\epsilon\to 0$ and $\delta\to 0$,
    \begin{equation*}
        I_1^{\epsilon,\delta}\to -\int_0^T \!\!\! \int_\Rd  n_\nu \Lap W_\nu \phi
			\dx x\dx t- \int_0^T\int_\Rd n_\nu\nabla W_\nu\cdot\nabla\phi \dx x \dx t. 
    \end{equation*}
    Now we look at the next term from
	Eq.~\eqref{eq:entropy_estimate_split}, namely
	\begin{equation*}
		\begin{aligned}
			I_{2}^{\epsilon, \delta}
			 & = \int_0^T \!\!\! \int_\Rd n\epsnu\log(n\epsnu+\delta)\grad
			\phi \cdot \grad W_\nu \dx x\dx t.
		\end{aligned}
	\end{equation*}
	Using the Dominated Convergence Theorem, in the limit $ \epsilon \to 0$ and  $\delta \to 0 $ we find that
	\begin{equation*}
		\begin{aligned}
			I^{\epsilon,\delta}_{2} \to \int_0^T \!\!\! \int_\Rd  n_\nu\log n_\nu \grad
			\phi \cdot
			\grad  W_\nu  \dx x \dx t.
		\end{aligned}
	\end{equation*}
	The next term from Eq.~\eqref{eq:entropy_estimate_split}, $I_{3}^{\epsilon,\delta}$, is nonnegative and can be dropped in the limit.  
    The fourth term from Eq.~\eqref{eq:entropy_estimate_split} is estimated by
	\begin{equation*}
		\begin{aligned}
			I_{4}^{\epsilon,\delta} &= \epsilon \int_0^T \!\!\! \int_\Rd \log(n\epsnu
			+\delta)\grad \phi \cdot \grad n\epsnu \dx x \dx t\\[0.4em]
            &\leq \prt*{\norm{n\epsnu}_{L^\infty(0,T; L^\infty(\Rd))}+|\log\delta|}\norm{\nabla\phi}_{L^2(\Rd)}\norm{\nabla n\epsnu}_{L^2(0,T;L^2(\Rd))}\\[0.4em]
            &\leq C\sqrt\epsilon,
		\end{aligned}
	\end{equation*}
	having used Lemma~\ref{prop:regular_n_bounds}. Thus,
	when $ \epsilon \to 0$ we have $ I_{4}^{\epsilon,\delta} \to 0 $. The final term from Eq.~\eqref{eq:entropy_estimate_split} is given by
	\begin{equation*}
		\begin{aligned}
			I_{5}^{\epsilon, \delta} = \int_0^T \!\!\! \int_\Rd \log(n\epsnu+\delta) \prt*{n\1\epsnu G^{(1)}(n_\nu) + n\2\epsnu G^{(2)}(n_\nu)} \phi \dx x \dx t,
		\end{aligned}
	\end{equation*}
	which, again by Dominated Convergence, converges to
	\begin{equation*}
		\begin{aligned}
			I_{5}^{\epsilon, \delta} &\to \int_0^T \!\!\! \int_\Rd \log n_\nu\prt*{
				n_\nu^{(1)} G^{(1)}(n_\nu) + n_\nu^{(2)} G^{(2)}(n_\nu)} \phi \dx x \dx t.
		\end{aligned}
	\end{equation*}
	Finally, since $ n\epsnu \in C([0,T];L^{2}(\Rd))$, the mapping $ t \mapsto \mathcal{H}^{\phi}_{\delta} [n\epsnu(t)] $ is continuous in $[0,T]$. We therefore have
    \begin{align*}
        \int_0^T \dfrac{\mathup{d}{}}{\mathup{dt}} \mathcal{H}^{\phi}_{\delta} [n\epsnu]
			\dx t= \mathcal{H}^{\phi}_{\delta} [n\epsnu](T) - \lim_{s\to 0}\mathcal{H}^{\phi}_{\delta} [n\epsnu](s) = \mathcal{H}^{\phi}_{\delta} [n\epsnu](T) - \mathcal{H}^{\phi}_{\delta} [n\epsnu^{\mathrm{in}}],
    \end{align*}
    and, as $\epsilon\to 0$ and then $\delta\to0$, we obtain
    \begin{align*}
        \mathcal{H}^{\phi}_{\delta} [n\epsnu](T) &- \mathcal{H}^{\phi}_{\delta} [n\epsnu^{\mathrm{in}}]\\ &\to \int_\Rd n_\nu(T)\prt{\log(n_\nu(T)+\delta)-1}\phi \dx x-\int_\Rd n_\nu^{\mathrm{in}}\prt{\log(n_\nu^{\mathrm{in}}+\delta)-1}\phi \dx x\\
        &\to \int_\Rd n_\nu(T)\prt{\log n_\nu(T)-1}\phi \dx x-\int_\Rd n_\nu^{\mathrm{in}}\prt{\log n_\nu^{\mathrm{in}}-1}\phi \dx x.
    \end{align*}
    
 Combining all the pieces of the jigsaw, we can pass to the limits $\epsilon\to 0$ and $\delta\to 0$ in Eq.~\eqref{eq:entropy_estimate_split} to get 
	\begin{equation*}
		\begin{aligned}
			\mathcal{H}^{\phi}[n_\nu](T) - \mathcal{H}^{\phi}[n_\nu^{\mathrm{in}}]
			 & - \int_0^T \!\!\! \int_\Rd  n_\nu \Lap W_\nu \phi \dx x \dx t-
			\int_0^T \!\!\! \int_\Rd  n_\nu \grad W_\nu \cdot \grad \phi \dx x
			\dx t \\
			 &\;\; + \int_0^T \!\!\! \int_\Rd  n_\nu\log n_\nu\grad \phi \cdot \grad W_\nu\dx x \dx t \\
			 & \leq \int_0^T \!\!\! \int_\Rd \log n_\nu\left[
			n^{(1)}_{\nu}G^{(1)}(n_\nu) +n^{(2)}_{\nu}G^{(2)}(n_\nu) \right] \phi \dx x \dx t,
		\end{aligned}
	\end{equation*}
	for any $ \phi \in C_{c}^{\infty}(\Rd),\, \phi \geq 0 $. 
    Let us now choose $\phi=\chi_R$, where $\chi_R$ is a sequence of smooth cut-off functions such that $|\nabla \chi_R|\lesssim R^{-1}$. Then, using the $L^\infty L^1$-control of $n_\nu\log n_\nu$, we can pass to the limit $R\to\infty$ by the Monotone Convergence Theorem to obtain Eq.~\eqref{eq:entropy_inequality}. \qedhere
\end{proof}

Let us conclude this section with a short remark that, by virtue of Proposition~\ref{lem:second_moment_nu}, the entropy as well as the right-hand side of Eq.~\eqref{eq:entropy_inequality} are bounded uniformly in $\nu>0$.

\section[Compactness]{Compactness of $ W_\nu $ }
This section is dedicated to establishing the strong convergence of $W_\nu$ in $L^2(0,T; L^2(\Rd))$ using the celebrated Aubin-Lions Lemma.

\begin{proposition}[Regularity of $W_\nu$]
    \label{prop:regulartiy-W}
    Let $(W_\nu)_{\nu>0}$ be a sequence of potentials associated to solutions of System \eqref{eq:Brinkman}. Then, there exists a constant $C>0$ such that
    \begin{align*}
    	\norm{\grad W_\nu}_{L^{2}(0,T;L^{2}(\Rd))} \leq C, \qquad \text{and}\qquad
    	 \norm{\partial_{t} W_\nu}_{L^{2}(0,T;H^{-1}(\Rd))} \leq C.
    \end{align*}
\end{proposition}
\begin{remark}[Weak Convergence of Velocity]
    It is worthwhile to point out that the uniform bound on $\nabla W_\nu$ implies its weak convergence (up to a subsequence) which plays an important role in the subsequent analysis.
\end{remark}

\begin{proof}
    We begin by proving the spatial regularity. To this end, we observe that
	\begin{align*}
        \int_0^T \!\!\! \int_\Rd \abs{\grad W_\nu}^{2} \dx x \dx t
        & = - \int_0^T \!\!\! \int_\Rd W_\nu \Delta  W_\nu  \dx x \dx t.
    \end{align*}
	Using Brinkman's law, we get
	\begin{align*}
        \int_0^T \!\!\! \int_\Rd \abs{\grad W_\nu}^{2} \dx x \dx t
        & = - \int_0^T \!\!\! \int_\Rd n_\nu \Lap W_\nu  \dx x
        \dx t- \nu \int_0^T \!\!\! \int_\Rd  \abs{\Lap W_\nu}^{2}  \dx x \dx t \\
        & \leq - \int_0^T \!\!\! \int_\Rd n_\nu \Lap W_\nu  \dx x \dx t.
	\end{align*}
    Rearranging the entropy dissipation inequality, Eq.~\eqref{eq:entropy_inequality}, we find
    \begin{equation}\label{eq: unif bnd nDeltaW}
    \begin{aligned}
       - \int_0^T \!\!\! \int_\Rd n_\nu \Lap W_\nu  \dx x \dx t & \leq  \int_0^T \!\!\! \int_\Rd \log n_\nu \brk*{ n^{(1)}_{\nu}G^{(1)}(n_\nu) +n^{(2)}_{\nu}G^{(2)}(n_\nu)} \dx x \dx t\\[0.5em]
       &\qquad + \curlyH[n_\nu^{\mathrm{in}}] - \curlyH[n_\nu](T)\\[0.7em]
       &\leq C\prt*{\norm{n_\nu}_{L^\infty(0,T;L^1(\Rd))} + \norm{n_\nu\log n_\nu}_{L^\infty(0,T;L^1(\Rd))}},
    \end{aligned}
    \end{equation}
    which is bounded uniformly. Hence $ \grad  W_\nu \in L^2(0,T;L^2(\Rd))$.
As for the bound on the time derivative, let us recall that $W_\nu = K_\nu \star n_\nu$, and $\norm{K_\nu}_{L^1(\Rd)}=~1$. Then, for any $\varphi\in L^2(0,T;H^1(\Rd))$, we have
    \begin{equation*}
		\begin{aligned}
			\int_0^T \!\!\! \int_\Rd  \varphi  \partialt{W_\nu} \dx x \dx t
			 & = \int_0^T \!\!\! \int_\Rd  \partialt{n_\nu} K_\nu\star\varphi \dx x \dx t\\[0.5em]
             & = -\int_0^T\!\!\!\int_\Rd n_\nu\nabla W_\nu \cdot \nabla\varphi\star K_\nu\dx x\dx t \\[0.5em]
            &\qquad + \int_0^T\!\!\!\int_\Rd \brk*{n_\nu^{(1)}G^{(1)}(n_\nu) + n_\nu^{(2)}G^{(2)}(n_\nu)}\varphi\star K_\nu\dx x\dx t\\[0.8em]
            &\leq \norm{\sqrt{n_\nu}}_{L^\infty(0,T;L^\infty(\Rd))}\norm{\sqrt{n_\nu}\nabla W_\nu}_{L^2(0,T; L^2(\Rd))}\norm{\nabla\varphi}_{L^2(0,T; L^2(\Rd))}\\[0.5em]
            &\qquad + C\norm{\varphi}_{L^2(0,T; L^2(\Rd))}.
		\end{aligned}
	\end{equation*}
    It follows that $\partial_t W_\nu \in L^2(0,T;H^{-1}(\Rd))$, uniformly in $\nu>0$. \qedhere
\end{proof}
By Proposition~\ref{prop:regulartiy-W}, we deduce that there exists a $ \chi \in L^{2}(0,T;L^{2}(\Rd)) $
such that
\begin{equation*}
	\begin{aligned}
		W_\nu \to \chi,
	\end{aligned}
\end{equation*}
strongly in $ L^2(0,T;L^{2}(\Rd)) $ as $ \nu \to 0 $. It is not difficult to identify the limit of $W_\nu$ as the weak limit of $n_\nu$, $n_0$.
\begin{proposition}[Identification of limit]\label{prop:Identification of limit}
	Let $ n_{0}$ be the weak limit of $n_\nu$. Then, the strong limit of $  W_\nu $
	in $ L^2(0,T;L^{2}(\Rd)) $ is $ n_{0} $.
\end{proposition}
\begin{proof}
	Let $ \varphi \in C^{\infty }_{c}(\Rd\times(0,T))$ be a test function.
	Then
	\begin{equation*}
		\begin{aligned}
			\int_0^T \!\!\! \int_\Rd  W_\nu \varphi  \dx x \dx t\rightarrow \int_0^T
			\int_\Rd  \chi \varphi  \dx x \dx t.
		\end{aligned}
	\end{equation*}
	Concomitantly, we know,
	\begin{equation*}
		\begin{aligned}
			\int_0^T \!\!\! \int_\Rd W_\nu \varphi \dx x \dx t
			 & = \int_0^T \!\!\! \int_\Rd K_{\nu}\star n_\nu \varphi\dx x \dx t 
			  = \int_0^T \!\!\! \int_\Rd n_\nu K_{\nu}\star\varphi  \dx x \dx t.
		\end{aligned}
	\end{equation*}
    Combining the two equations, we get
	\begin{equation*}
		\begin{aligned}
			\int_0^T \!\!\! \int_\Rd  n_\nu K_{\nu}\star\varphi \dx x \dx t\to \int_0^T
			\int_\Rd  \chi \varphi \dx x\dx t.
		\end{aligned}
	\end{equation*}
    On the other hand, we know that $ n_\nu \rightharpoonup n_0 $ in $ L^{2}
		(0,T;L^{2}(\Rd)) $ and $ K_{\nu}\star g \to g $ in $ L^{2}(0,T;L^{2}(\Rd)) $, as
	$ \nu \to 0 $. Thus,
	\begin{equation*}
		\begin{aligned}
			\int_0^T \!\!\! \int_\Rd  n_\nu K_{\nu}\star g \dx x \dx t\to \int_0^T
			\int_\Rd  n_{0} g \dx x\dx t.
		\end{aligned}
	\end{equation*}
	This shows that $ \chi = n_0 $. \qedhere
    \end{proof}
Using the entropy inequality, the strong compactness of $W_\nu$ can be leveraged to get the compactness of the total densities $n_\nu$. 

\begin{proposition}
    \label{lem:compactness_n}
    We have $ n_\nu \to n_{0}  $ in $ L^{2}(0,T;L^{2}(\Rd)) $.
\end{proposition}
\begin{proof}
    Using Brinkman's law and the entropy inequality, \cf \eqref{eq: unif bnd nDeltaW}, we have
    	\begin{equation*}
		\begin{aligned}
			\int_0^T\int_\Rd |n_\nu -W_\nu|^2 \dx x \dx t &=
			 \int_0^T\int_\Rd(n_\nu -W_\nu ) (n_\nu -W_\nu ) \dx x\dx t                   \\
			&= - \nu \int_0^T \!\!\! \int_\Rd  n_\nu \Lap W_\nu  \dx x \dx t+ \nu
			\int_0^T \!\!\! \int_\Rd  W_\nu \Lap W_\nu  \dx x \dx t                           \\
			&= - \nu \int_0^T \!\!\! \int_\Rd  n_\nu \Lap W_\nu  \dx x \dx t-
			\nu \int_0^T \!\!\! \int_\Rd  \abs{\grad W_\nu}^{2} \dx x \dx t\\
            &\leq C\nu. 
		\end{aligned}
	\end{equation*}
    Consequently,
    \begin{align*}
        \norm{n_\nu - n_{0} }_{L^{2}(0,T;L^{2}(\Rd)) }
			 & \leq \norm{n_\nu -W_\nu }_{L^{2}(0,T;L^{2}(\Rd))} +
			\norm{W_\nu - n_{0}}_{L^{2}(0,T;L^{2}(\Rd))}\\[0.5em] 
             & \leq C\sqrt{\nu} + \norm{W_\nu - n_{0}}_{L^{2}(0,T;L^{2}(\Rd))}.  
    \end{align*}
 
\end{proof}

\section{Passage to the Limit}
\label{sec:passing_to_the_limit}
In this section, we conclude the passage to the limit $\nu\to0$. First, having established the strong compactness of the total density $n_\nu$, we can pass to the limit in the equation for $n_\nu$ to obtain (in the distributional sense)
\begin{equation}
\label{eq:totallimit}
		\partialt {n_{0}} - \div \left( n_{0} \grad n_{0}  \right) = n_{0}^{(1)}G^{(1)}(n_{0} ) + n_{0}^{(2)}G^{(2)}(n_{0} ),
\end{equation}
where $ n^{(i)}_{0}  $ is the weak limit of $ n^{(i)}_{\nu} $. 

To pass to the limit in the equations for the individual densities we will use the entropy inequality~\eqref{eq:entropy_inequality} together with the corresponding entropy \emph{equality} satisfied by the solution to Eq.~\eqref{eq:totallimit} to upgrade the weak convergence of the velocity $\nabla W_\nu$ to strong convergence.
From the previous estimates we have
\begin{equation*}
	\begin{aligned}
		n_{0} \in C_w\left( [0,T];L^{2}(\Rd) \right) \cap L^{2}(0,T;H^{1}(\Rd)) \cap
		L^{\infty}(0,T;L^1(\R^d)\cap L^\infty(\Rd)).
	\end{aligned}
\end{equation*}
This regularity is enough to guarantee the entropy equality
\begin{equation}
\label{eq:entropy_inequality_n0}
	\mathcal{H}[n_0](T) + \int_0^T \!\!\! \int_\Rd  \abs{\grad n_0}^{2} \dx x \dx t- R\prt*{n_{0}^{(1)}, n_{0}^{(2)}, n_{0}}
	=
	\mathcal{H}[n_0^{\mathrm{in}}],
\end{equation}
where
\begin{align*}
	R\prt*{n_{0}^{(1)}, n_{0}^{(2)}, n_{0}} := \int_0^T \!\!\! \int_\Rd \log n_0 \left[ n_{0}^{(1)}G^{(1)}(n_{0}) +n_{0}^{(2)}G^{(2)}(n_{0}) \right]  \dx x \dx t.
\end{align*}
Making use of the assumption
\begin{align*}
	\mathcal{H}[n_\nu^\mathrm{in}] \leq \mathcal{H}[n_{0}^\mathrm{in}],
\end{align*}
we can compare the two entropy inequalities~\eqref{eq:entropy_inequality} and~\eqref{eq:entropy_inequality_n0}.
Let
\begin{align*}
 	R\prt*{n_\nu^{(1)}, n_\nu^{(2)}, n_\nu} := \int_0^T \!\!\! \int_\Rd \log n_\nu\left[ n_\nu^{(1)}G^{(1)}(n_\nu) +n_\nu^{(2)}G^{(2)}(n_\nu) \right]  \dx x \dx t.
 \end{align*}
Upon rearranging the two inequalities we get
\begin{align}
    \label{eq:entropies-stacked}
    \begin{split}
	- \int_0^T \!\!\! \int_\Rd  n_\nu \Lap W_\nu \dx x \dx t
	 & \leq  \mathcal{H}[n_\nu^\mathrm{in}] - \mathcal{H}[n_\nu](T) + R\prt*{n_\nu^{(1)}, n_\nu^{(2)}, n_\nu} \\[0.5em]
	 & \leq  \mathcal{H}[n_{0}^\mathrm{in}] - \mathcal{H}[n_\nu](T) + R\prt*{n_\nu^{(1)}, n_\nu^{(2)}, n_\nu}   \\[0.5em]
	 & = \int_0^T\int_\Rd \abs*{\nabla n_0}^2 \dx x \dx t + \mathcal{E}_\nu,
  \end{split}
\end{align}
where, for convenience, we use the notation
\begin{align*}
	\mathcal{E}_\nu := \mathcal{H}[n_{0}](T) - \mathcal{H}[n_\nu](T) + R\prt*{n_\nu^{(1)}, n_\nu^{(2)}, n_\nu} - R\prt*{n_{0}^{(1)}, n_{0}^{(2)}, n_{0}}.
\end{align*}

Recall that
\begin{align*}
	\int_0^T \!\!\! \int_\Rd \abs*{\nabla W_\nu}^2 \dx x \dx t & = -\int_0^T \!\!\! \int_\Rd W_\nu  \Lap W_\nu  \dx x \dx t \\
    & = - \int_0^T \!\!\! \int_\Rd n_\nu  \Lap W_\nu  \dx x \dx t + \int_0^T \!\!\! \int_\Rd (n_\nu-W_\nu) \Lap W_\nu \dx x \dx t \\
    & = - \int_0^T \!\!\! \int_\Rd n_\nu  \Lap W_\nu  \dx x \dx t -\frac{1}{\nu} \int_0^T \!\!\! \int_\Rd \abs*{n_\nu-W_\nu}^2 \dx x \dx t  \\
    & \leq - \int_0^T \!\!\! \int_\Rd n_\nu  \Lap W_\nu  \dx x \dx t.
\end{align*}
This, in conjunction with Eq. \eqref{eq:entropies-stacked}, gives
\begin{align}
\label{eq:uppersemicts}
    \int_0^T \!\!\! \int_\Rd \abs*{\nabla W_\nu}^2 \dx x \dx t \leq \int_0^T \!\!\! \int_\Rd |\nabla n_0|^2 \dx x \dx t + \mathcal{E}_{\nu}.
\end{align}
We will now investigate the term $\mathcal{E}_{\nu}$ and show that it vanishes in the limit $\nu\to0$.

Given the uniform $L^1$ bound on $n_\nu(T)$ and uniform control of the second moment, we infer that the sequence $n_\nu(T)$ converges to $n_0(T)$ weakly in $L^1(\Rd)$. Therefore, by the convexity of the entropy functional, we have
\begin{equation*}
    \mathcal{H}[n_{0}](T) \leq \liminf_{\nu\to 0}\mathcal{H}[n_\nu](T),
\end{equation*}
or equivalently
\begin{equation*}
    \limsup_{\nu\to 0}\prt*{\mathcal{H}[n_{0}](T) - \mathcal{H}[n_\nu](T)} \leq 0.
\end{equation*}

Let us now focus on the reaction terms. We shall use the tightness of the sequence $(n_\nu)_{\nu}$, as well as its pointwise convergence to $n_0$, to split the domain of integration into a subdomain on which the terms corresponding to $n_\nu$ converge to their counterparts for $n_0$ and remainders which can be made arbitrarily small.

Let $\varepsilon>0$ and $0<\beta<e^{-1}$. By tightness, there exists a set $K\subset \Rd\times[0,T]$ of finite measure such that
\begin{equation}
    \sup_{\nu>0} \iint_{K^c} n_\nu \leq \varepsilon.
\end{equation}
Furthermore, by Egorov's Theorem there exists $A\subset K$ with $\mathrm{meas}(A)<\varepsilon$ such that $n_\nu$ converges uniformly to $n_0$ on $K\setminus A$. Then, for $\nu$ small enough, we have
\begin{equation*}
    n_\nu \geq \beta \implies n_0 \geq \frac{\beta}{2}.
\end{equation*}

Writing
\begin{equation*}
    \int_0^T\!\!\!\int_\Rd n\i_\nu\log n_\nu G\i(n_\nu)\dx x \dx t = \iint_{K^c} n\i_\nu\log n_\nu G\i(n_\nu)\dx x \dx t + \iint_K n\i_\nu\log n_\nu G\i(n_\nu)\dx x \dx t,
\end{equation*}
we can guarantee that the first integral is small. Indeed, letting $\bar n\i$ be the zero of $G\i$, we have
\begin{align*}
    \iint_{K^c\cap\set{n_\nu<\bar n\i}} n\i_\nu\log n_\nu G\i(n_\nu)\dx x \dx t &\leq \iint_{K^c\cap\set{n_\nu<\bar n\i}} n\i_\nu(n_\nu-1) G\i(n_\nu)\dx x \dx t\\
    &\leq \prt*{\norm{G\i}_{L^\infty(\R)}+\norm{n_\nu}_{L^\infty(\Rd)}} \iint_{K^c}n_\nu\dx x \dx t\\
    &\leq C\varepsilon,
\end{align*}
and
\begin{align*}
    \iint_{K^c\cap\set{n_\nu\geq\bar n\i}} & n\i_\nu\log n_\nu G\i(n_\nu)\dx x \dx t\\
    &\leq \norm{G\i}_{L^\infty(\R)}\iint_{K^c\cap\set{n_\nu<\bar n\i}} n_\nu|\log n_\nu|\dx x \dx t\\
    &\leq \norm{G\i}_{L^\infty(\R)}\iint_{K^c\cap\set{n_\nu\geq\bar n\i}} \prt{\sqrt{n_\nu}+n_\nu^2}\dx x \dx t\\
    &\leq \norm{G\i}_{L^\infty(\R)}\prt*{\norm{n_\nu}_{L^\infty(0,T;L^\infty(\Rd))}+\frac{1}{\sqrt{\bar n\i}}}\iint_{K^c} n_\nu\dx x \dx t\\
    &\leq C\varepsilon.
\end{align*}

It remains to treat the integral over $K$. This we split into three disjoint subdomains as follows
\begin{equation*}
    K = \prt*{K\setminus A \cap \set{n_\nu\geq\beta}}\, \cup\, \prt*{K\setminus A \cap \set{n_\nu< \beta}}\,\cup\, A \equiv K_1 \cup K_2 \cup A. 
\end{equation*}
For the integral over $A$, we notice that
\begin{align*}
    \iint_A n\i_\nu\log n_\nu G\i(n_\nu)\dx x \dx t \leq \norm{G^{(i)}}_{L^{\infty}(\R)}\prt*{ 1+\norm{n_\nu}_{L^\infty(0,T;L^\infty(\Rd))}^{2}} \mathrm{meas}(A) \leq C\varepsilon,
\end{align*}
while the integral over $K_2$ is bounded as follows
\begin{align*}
    \iint_{K_2} n\i_\nu\log n_\nu G\i(n_\nu)\dx x \dx t \leq \norm{G^{(i)}}_{L^{\infty}(\R)}\mathrm{meas}(K)\beta|\log\beta| \leq C\beta|\log\beta|.
\end{align*}
Finally, on the set $K_1$, the sequence $\log n_\nu$ converges strongly to $\log n_0$ and therefore
\begin{equation*}
    \iint_{K_1} n\i_\nu\log n_\nu G\i(n_\nu)\dx x \dx t \to \iint_{K_1} n\i_0\log n_0 G\i(n_0)\dx x \dx t.
\end{equation*}

It follows that
\begin{align*}
    & R\prt*{n_\nu^{(1)}, n_\nu^{(2)}, n_\nu} - R\prt*{n_{0}^{(1)}, n_{0}^{(2)}, n_{0}} \\
    &\quad\leq \abs*{\iint_{K_1} n\1_\nu\log n_\nu G\1(n_\nu)\dx x \dx t - \iint_{K_1} n\1_0\log n_0 G\1(n_0)\dx x \dx t} \\
    &\qquad+ \abs*{\iint_{K_1} n\2_\nu\log n_\nu G\2(n_\nu)\dx x \dx t - \iint_{K_1} n\2_0\log n_0 G\2(n_0)\dx x \dx t} + C\prt*{\varepsilon + \beta|\log\beta|}.
\end{align*}

Passing to the limit in Eq.~\eqref{eq:uppersemicts}, we therefore deduce
\begin{align*}
    \limsup_{\nu\to0} \int_0^T\!\!\!\int_\Rd\abs*{\nabla W_\nu}^2 \dx x \dx t
    &\leq \int_0^T \!\!\! \int_\Rd |\nabla n_0|^2 \dx x \dx t + \limsup_{\nu\to0}\mathcal{E}_{\nu}\\
    &\leq \int_0^T \!\!\! \int_\Rd |\nabla n_0|^2 \dx x \dx t + C\prt*{\varepsilon + \beta|\log\beta|}.
\end{align*}
Since $\varepsilon$ and $\beta$ can be chosen arbitrarily small, we deduce the desired upper semicontinuity property
\begin{equation}
\label{eq:nablaWstrong}
    \limsup_{\nu\to0} \int_0^T\!\!\!\int_\Rd\abs*{\nabla W_\nu}^2 \dx x \dx t\\
    \leq \int_0^T \!\!\! \int_\Rd |\nabla n_0|^2 \dx x \dx t.
\end{equation}
On the other hand, from the weak convergence $\nabla W_\nu \rightharpoonup \nabla n_0$ in $L^2(0,T;L^2(\Rd))$ we have 
\begin{equation*}
        \liminf_{\nu\to0} \int_0^T\!\!\!\int_\Rd\abs*{\nabla W_\nu}^2 \dx x \dx t
    \geq \int_0^T \!\!\! \int_\Rd |\nabla n_0|^2 \dx x \dx t.
\end{equation*}
Putting the two inequalities together we deduce
\begin{equation*}
        \lim_{\nu\to0} \int_0^T\!\!\!\int_\Rd\abs*{\nabla W_\nu}^2 \dx x \dx t = \int_0^T \!\!\! \int_\Rd |\nabla n_0|^2 \dx x \dx t.
\end{equation*}
It follows that $\nabla W_\nu \to \nabla n_0$, strongly in $L^2(0,T;L^2(\Rd))$.
Armed with the strong convergence of the velocities, we can now pass to the limit in System~\eqref{eq:Brinkman} to obtain
\begin{equation*}
    \partial_{t}n^{(i)}_{0} + \div \left(n^{(i)}_{0} \grad n_{0}  \right) = n_{0}^{(i)}G^{(i)}(n_{0} ),\qquad i=1,2.
\end{equation*}
This concludes the last step of the proof of Theorem~\ref{thm:inviscid_limit}.

\section{A Short Remark on Phenotypically Stratified Models}
Let us conclude our analysis with a short remark on an interesting extension of the model. In recent years, there has been a growing interest in structured models for tissue growth. As the aggressiveness of a tumour is intimately connected to the cells' phenotypic traits which, in turn, influence several characteristics, such as cell motility or cell division rate, a phenotypic variable, $y\in[0,1]$, may be introduced to account for the population's heterogeneity. In~\cite{Dav2022}, the author shows the existence of weak solutions to 
\begin{align*}
    \partialt {n_0}(y, x, t) = \nabla\cdot(n_0\nabla N_0) + n_0G(y, N_0),
\end{align*}
where all operators act on the $x$-variable,
which can be understood as a structured counterpart of System \eqref{eq:Darcy}. Here, the total population density is given by the integral of the phenotypic densities along the interval of traits, \ie,
\begin{align*}
    N_0(x,t) := \int_0^1 n_0(y, x, t) \dx y,
\end{align*}
acting as the infinite-dimensional  counterpart of $n_0 = n_0\1 + n_0\2$. In the context of viscoelastic tissues, we propose a structured Brinkman model of the form
\begin{align*}
	\left\{
	\begin{array}{rl}
		\ds \partialt{n_\nu}(y,x,t) - \nabla \cdot \prt*{n_\nu \grad W_\nu} \!\!\!& = \ds n_\nu G(y, N_\nu), \\[1em]
		\ds -\nu \Lap W_\nu +W_\nu \!\!\! & = \ds  N_\nu,
	\end{array}
	\right.
\end{align*}
where, again, all operators act in physical space. Formally, we obtain the entropy dissipation inequality of the total population, \ie, 
\begin{align*}
    \curlyH[N_\nu](T) - \curlyH[N_\nu^{\mathrm{in}}]
		  - \int_0^T \!\!\! \int_\Rd  N_\nu \Lap W_\nu \dx x \dx t \leq \int_0^T \!\!\! \int_\Rd \log N_\nu\int_0^1 n_\nu(y) G(y,N_\nu) \dx y \dx x \dx t,
\end{align*}
paralleling that of Lemma \ref{lem:entropy_inequality}. Using a similar argument as in the main part of the paper, we expect to obtain
\begin{align*}
    \lim_{\nu \to 0} \int_0^T \!\!\! \int_\Rd \abs*{\nabla W_\nu}^2 \dx x \dx t = \int_0^T \!\!\! \int_\Rd \abs*{\nabla N_0}^2 \dx x \dx t,
\end{align*}
which gives strong reason to believe that our strategy is applicable in other contexts.

\section*{Acknowledgements}
T.D. would like to acknowledge the hospitality of Technische Universit\"at Dresden during his research stay, where the foundation for this project was laid. M.S. is happy to acknowledge the fruitful discussions during his time as visiting professor at the University of Warsaw during the Thematic Research Programme `PDEs in Mathematical Biology' (Excellence Initiative Research University).
 
\bibliographystyle{dinat}

\begin{thebibliography}{59}
\makeatletter
\newcommand{\dinatlabel}[1]%
{\ifNAT@numbers\else\NAT@biblabelnum{#1}\hspace{2\labelsep}\fi}
\makeatother
\expandafter\ifx\csname natexlab\endcsname\relax\def\natexlab#1{#1}\fi
\expandafter\ifx\csname url\endcsname\relax\def\url#1{\texttt{#1}}\fi

\bibitem[Alasio u.\,a.(2022)Alasio, Bruna, Fagioli und Schulz]{ABFS2022}
\dinatlabel{Alasio u.\,a. 2022} \textsc{Alasio}, Luca~; \textsc{Bruna}, Maria~;
  \textsc{Fagioli}, Simone~; \textsc{Schulz}, Simon:
\newblock Existence and regularity for a system of porous medium equations with
  small cross-diffusion and nonlocal drifts.
\newblock In: \emph{Nonlinear Analysis}
\newblock 223 (2022), S.~113064

\bibitem[Alasio u.\,a.(2020)Alasio, Ranetbauer, Schmidtchen und
  Wolfram]{ARSW2020}
\dinatlabel{Alasio u.\,a. 2020} \textsc{Alasio}, Luca~; \textsc{Ranetbauer},
  Helene~; \textsc{Schmidtchen}, Markus~; \textsc{Wolfram}, Marie-Therese:
\newblock Trend to equilibrium for systems with small cross-diffusion.
\newblock In: \emph{ESAIM: Mathematical Modelling and Numerical Analysis}
\newblock 54 (2020), Nr.~5, S.~1661--1688



\bibitem[Allaire(1991)]{All91}
\dinatlabel{Allaire 1991} \textsc{Allaire}, Grégoire:
\newblock Homogenization of the {N}avier-{S}tokes equations and derivation of {
  B}rinkman's law.
\newblock In: \emph{Math\'{e}matiques appliqu\'{e}es aux sciences de
  l'ing\'{e}nieur ({S}antiago, 1989)}.
\newblock C\'{e}padu\`es, Toulouse, 1991, S.~7--20


\bibitem[Belgacem,  u.\,a. (2013) Belgacem und Jabin]{BelgacemJabin}
\dinatlabel{Belgacem,  u.\,a. 2013} \textsc{Belgacem}, Fethi Ben~; \textsc{Jabin},
  Pierre-Emmanuel~:
\newblock Compactness for nonlinear continuity equations.
\newblock In: \emph{Journal of Functional Analysis}
\newblock 264 (2013), Nr.~1, S.~139--168


\bibitem[Berendsen u.\,a.(2017)Berendsen, Burger und Pietschmann]{BBP2017}
\dinatlabel{Berendsen u.\,a. 2017} \textsc{Berendsen}, Judith~;
  \textsc{Burger}, Martin~; \textsc{Pietschmann}, Jan-Frederik:
\newblock On a cross-diffusion model for multiple species with nonlocal
  interaction and size exclusion.
\newblock In: \emph{Nonlinear Analysis}
\newblock 159 (2017), S.~10--39

\bibitem[Bertsch u.\,a.(2010)Bertsch, Dal~Passo und Mimura]{BDPM2010}
\dinatlabel{Bertsch u.\,a. 2010} \textsc{Bertsch}, M.~; \textsc{Dal~Passo},
  R.~; \textsc{Mimura}, M.:
\newblock A free boundary problem arising in a simplified tumour growth model
  of contact inhibition.
\newblock In: \emph{Interfaces and Free Boundaries}
\newblock 12 (2010), Nr.~2, S.~235--250

\bibitem[Bertsch u.\,a.(2012)Bertsch, Hilhorst, Izuhara und Mimura]{BHIM2012}
\dinatlabel{Bertsch u.\,a. 2012} \textsc{Bertsch}, M.~; \textsc{Hilhorst}, D.~;
  \textsc{Izuhara}, H.~; \textsc{Mimura}, M.:
\newblock A nonlinear parabolic-hyperbolic system for contact inhibition of
  cell-growth.
\newblock In: \emph{Differ. Equ. Appl}
\newblock 4 (2012), Nr.~1, S.~137--157

\bibitem[Bertsch u.\,a.(2015)Bertsch, Hilhorst, Izuhara, Mimura und
  Wakasa]{BHIMW2015}
\dinatlabel{Bertsch u.\,a. 2015} \textsc{Bertsch}, M~; \textsc{Hilhorst}, D~;
  \textsc{Izuhara}, H~; \textsc{Mimura}, M~; \textsc{Wakasa}, T:
\newblock Travelling wave solutions of a parabolic-hyperbolic system for
  contact inhibition of cell-growth.
\newblock In: \emph{European Journal of Applied Mathematics}
\newblock 26 (2015), Nr.~3, S.~297--323

\bibitem[Bertsch u.\,a.(1987{\natexlab{a}})Bertsch, Gurtin und
  Hilhorst]{BGH1987}
\dinatlabel{Bertsch u.\,a. 1987{\natexlab{a}}} \textsc{Bertsch}, Michiel~;
  \textsc{Gurtin}, Morton~E.~; \textsc{Hilhorst}, Danielle:
\newblock On a degenerate diffusion equation of the form $c(z)_t= \varphi
  (z_x)_x$ with application to population dynamics.
\newblock In: \emph{Journal of differential equations}
\newblock 67 (1987), Nr.~1, S.~56--89

\bibitem[Bertsch u.\,a.(1987{\natexlab{b}})Bertsch, Gurtin und
  Hilhorst]{BGH1987a}
\dinatlabel{Bertsch u.\,a. 1987{\natexlab{b}}} \textsc{Bertsch}, Michiel~;
  \textsc{Gurtin}, Morton~E.~; \textsc{Hilhorst}, Danielle:
\newblock On interacting populations that disperse to avoid crowding: the case
  of equal dispersal velocities.
\newblock In: \emph{Nonlinear Analysis: Theory, Methods \& Applications}
\newblock 11 (1987), Nr.~4, S.~493--499

\bibitem[Bertsch u.\,a.(1985)Bertsch, Gurtin, Hilhorst und Peletier]{BGHP1985}
\dinatlabel{Bertsch u.\,a. 1985} \textsc{Bertsch}, Michiel~; \textsc{Gurtin},
  Morton~E.~; \textsc{Hilhorst}, Danielle~; \textsc{Peletier}, LA:
\newblock On interacting populations that disperse to avoid crowding:
  preservation of segregation.
\newblock In: \emph{Journal of mathematical biology}
\newblock 23 (1985), Nr.~1, S.~1--13

\bibitem[Bertsch u.\,a.(2020)Bertsch, Hilhorst, Izuhara, Mimura und
  Wakasa]{BHIHMW2020}
\dinatlabel{Bertsch u.\,a. 2020} \textsc{Bertsch}, Michiel~; \textsc{Hilhorst},
  Danielle~; \textsc{Izuhara}, Hirofumi~; \textsc{Mimura}, Masayasu~;
  \textsc{Wakasa}, Tohru:
\newblock A nonlinear parabolic-hyperbolic system for contact inhibition and a
  degenerate parabolic fisher kpp equation.
\newblock In: \emph{Discrete and Continuous Dynamical Systems}
\newblock 40 (2020), Nr.~6, S.~3117--3142

\bibitem[Bruna u.\,a.(2021{\natexlab{a}})Bruna, Burger und Carrillo]{BBC2021}
\dinatlabel{Bruna u.\,a. 2021{\natexlab{a}}} \textsc{Bruna}, Maria~;
  \textsc{Burger}, Martin~; \textsc{Carrillo}, José~A.:
\newblock Coarse graining of a Fokker–Planck equation with excluded volume
  effects preserving the gradient flow structure.
\newblock In: \emph{European Journal of Applied Mathematics}
\newblock 32 (2021), Nr.~4, S.~711–745

\bibitem[Bruna u.\,a.(2021{\natexlab{b}})Bruna, Burger, Pietschmann und
  Wolfram]{BBPW2021}
\dinatlabel{Bruna u.\,a. 2021{\natexlab{b}}} \textsc{Bruna}, Maria~;
  \textsc{Burger}, Martin~; \textsc{Pietschmann}, Jan-Frederik~;
  \textsc{Wolfram}, Marie-Therese:
\newblock Active crowds.
\newblock In: \emph{Active Particles, Volume 3: Advances in Theory, Models, and
  Applications}.
\newblock Springer, 2021, S.~35--73

\bibitem[Bruna u.\,a.(2022)Bruna, Chapman und Schmidtchen]{BCS2022}
\dinatlabel{Bruna u.\,a. 2022} \textsc{Bruna}, Maria~; \textsc{Chapman}, S~J.~;
  \textsc{Schmidtchen}, Markus:
\newblock Derivation of a macroscopic model for Brownian hard needles.
\newblock In: \emph{arXiv preprint arXiv:2208.03056}
\newblock (2022)

\bibitem[Bruna und Chapman(2012{\natexlab{a}})]{BC2012a}
\dinatlabel{Bruna und Chapman 2012{\natexlab{a}}} \textsc{Bruna}, Maria~;
  \textsc{Chapman}, S~J.:
\newblock Diffusion of multiple species with excluded-volume effects.
\newblock In: \emph{The Journal of chemical physics}
\newblock 137 (2012), Nr.~20, S.~204116

\bibitem[Bruna und Chapman(2012{\natexlab{b}})]{BC2012}
\dinatlabel{Bruna und Chapman 2012{\natexlab{b}}} \textsc{Bruna}, Maria~;
  \textsc{Chapman}, S~J.:
\newblock Excluded-volume effects in the diffusion of hard spheres.
\newblock In: \emph{Physical Review E}
\newblock 85 (2012), Nr.~1, S.~011103

\bibitem[Bruna u.\,a.(2017)Bruna, Chapman und Robinson]{BCR17}
\dinatlabel{Bruna u.\,a. 2017} \textsc{Bruna}, Maria~; \textsc{Chapman},
  S.~J.~; \textsc{Robinson}, Martin:
\newblock Diffusion of particles with short-range interactions.
\newblock In: \emph{SIAM J. Appl. Math.}
\newblock 77 (2017), Nr.~6, S.~2294--2316. --
\newblock URL \url{https://doi.org/10.1137/17M1118543}. --
\newblock ISSN 0036-1399

\bibitem[Bubba u.\,a.(2019)Bubba, Perthame, Pouchol und Schmidtchen]{BPPS2019}
\dinatlabel{Bubba u.\,a. 2019} \textsc{Bubba}, Federica~; \textsc{Perthame},
  Beno{\^i}t~; \textsc{Pouchol}, Camille~; \textsc{Schmidtchen}, Markus:
\newblock Hele--Shaw Limit for a System of Two Reaction-(Cross-)Diffusion
  Equations for Living Tissues.
\newblock In: \emph{Archive for Rational Mechanics and Analysis}
\newblock (2019), Dec. --
\newblock ISSN 1432-0673

\bibitem[Burger u.\,a.(2020)Burger, Carrillo, Pietschmann und
  Schmidtchen]{BCPS2020}
\dinatlabel{Burger u.\,a. 2020} \textsc{Burger}, Martin~; \textsc{Carrillo},
  Jos{\'e}~A~; \textsc{Pietschmann}, Jan-Frederik~; \textsc{Schmidtchen},
  Markus:
\newblock Segregation effects and gap formation in cross-diffusion models.
\newblock In: \emph{Interfaces and Free Boundaries}
\newblock 22 (2020), Nr.~2, S.~175--203

\bibitem[Burger u.\,a.(2010{\natexlab{b}})Burger, Di~Francesco, Pietschmann und
  Schlake]{BDPS2010}
\dinatlabel{Burger u.\,a. 2010{\natexlab{b}}} \textsc{Burger}, Martin~;
  \textsc{Di~Francesco}, Marco~; \textsc{Pietschmann}, Jan-Frederik~;
  \textsc{Schlake}, B{\"a}rbel:
\newblock Nonlinear cross-diffusion with size exclusion.
\newblock In: \emph{SIAM Journal on Mathematical Analysis}
\newblock 42 (2010), Nr.~6, S.~2842--2871

\bibitem[Burger und Esposito(2022)]{BE2022}
\dinatlabel{Burger und Esposito 2022} \textsc{Burger}, Martin~;
  \textsc{Esposito}, Antonio:
\newblock \emph{Porous medium equation and cross-diffusion systems as limit of
  nonlocal interaction}.
\newblock 2022. --
\newblock URL \url{https://arxiv.org/abs/2202.05030}

\bibitem[Busenberg und Travis(1983)]{BT1983}
\dinatlabel{Busenberg und Travis 1983} \textsc{Busenberg}, Stavros~N.~;
  \textsc{Travis}, Curtis~C.:
\newblock Epidemic models with spatial spread due to population migration.
\newblock In: \emph{Journal of Mathematical Biology}
\newblock 16 (1983), Nr.~2, S.~181--198

\bibitem[Byrne und Drasdo(2009)]{BD09}
\dinatlabel{Byrne und Drasdo 2009} \textsc{Byrne}, Helen~; \textsc{Drasdo},
  Dirk:
\newblock Individual-based and continuum models of growing cell populations: a
  comparison.
\newblock In: \emph{J. Math. Biol.}
\newblock 58 (2009), Nr.~4-5, S.~657--687. --
\newblock ISSN 0303-6812

\bibitem[Carl(1971)]{Car1971}
\dinatlabel{Carl 1971} \textsc{Carl}, Ernest~A.:
\newblock Population control in arctic ground squirrels.
\newblock In: \emph{Ecology}
\newblock 52 (1971), Nr.~3, S.~395--413

\bibitem[Carrillo u.\,a.(2019{\natexlab{a}})Carrillo, Craig und Yao]{CCY2019}
\dinatlabel{Carrillo u.\,a. 2019{\natexlab{a}}} \textsc{Carrillo},
  Jos{\'e}~A.~; \textsc{Craig}, Katy~; \textsc{Yao}, Yao:
\newblock Aggregation-diffusion equations: dynamics, asymptotics, and singular
  limits.
\newblock In: \emph{Active Particles, Volume 2: Advances in Theory, Models, and
  Applications}
\newblock (2019), S.~65--108

\bibitem[Carrillo u.\,a.(2018{\natexlab{a}})Carrillo, Colombi und
  Scianna]{CCS2018}
\dinatlabel{Carrillo u.\,a. 2018{\natexlab{a}}} \textsc{Carrillo},
  Jos{\'e}~A.~; \textsc{Colombi}, Annachiara~; \textsc{Scianna}, Marco:
\newblock Adhesion and volume constraints via nonlocal interactions determine
  cell organisation and migration profiles.
\newblock In: \emph{Journal of theoretical biology}
\newblock 445 (2018), S.~75--91

\bibitem[Carrillo u.\,a.(2019{\natexlab{b}})Carrillo, Craig und
  Patacchini]{CCP2019}
\dinatlabel{Carrillo u.\,a. 2019{\natexlab{b}}} \textsc{Carrillo},
  Jos{\'e}~A.~; \textsc{Craig}, Katy~; \textsc{Patacchini}, Francesco~S.:
\newblock A blob method for diffusion.
\newblock In: \emph{Calculus of Variations and Partial Differential Equations}
\newblock 58 (2019), S.~1--53

\bibitem[Carrillo u.\,a.(2011)Carrillo, DiFrancesco, Figalli, Laurent und
  Slep{\v{c}}ev]{CDFLS2011}
\dinatlabel{Carrillo u.\,a. 2011} \textsc{Carrillo}, José~A.~;
  \textsc{DiFrancesco}, M.~; \textsc{Figalli}, A.~; \textsc{Laurent}, T.~;
  \textsc{Slep{\v{c}}ev}, D.:
\newblock Global-in-time weak measure solutions and finite-time aggregation for
  nonlocal interaction equations.
\newblock In: \emph{Duke Mathematical Journal}
\newblock 156 (2011), Nr.~2, S.~229--271

\bibitem[Carrillo u.\,a.(2023)Carrillo, Esposito und Wu]{CEW2023}
\dinatlabel{Carrillo u.\,a. 2023} \textsc{Carrillo}, José~A.~;
  \textsc{Esposito}, Antonio~; \textsc{Wu}, Jeremy Sheung-Him:
\newblock \emph{Nonlocal approximation of nonlinear diffusion equations}.
\newblock 2023. --
\newblock URL \url{https://arxiv.org/abs/2302.08248}

\bibitem[Carrillo u.\,a.(2018{\natexlab{b}})Carrillo, Fagioli, Santambrogio und
  Schmidtchen]{CFSS2017}
\dinatlabel{Carrillo u.\,a. 2018{\natexlab{b}}} \textsc{Carrillo}, José~A.~;
  \textsc{Fagioli}, Simone~; \textsc{Santambrogio}, Filippo~;
  \textsc{Schmidtchen}, Markus:
\newblock Splitting Schemes and Segregation in Reaction Cross-Diffusion
  Systems.
\newblock In: \emph{SIAM Journal on Mathematical Analysis}
\newblock 50 (2018), Nr.~5, S.~5695--5718

\bibitem[Carrillo u.\,a.(2018{\natexlab{c}})Carrillo, Huang und
  Schmidtchen]{CHS2017}
\dinatlabel{Carrillo u.\,a. 2018{\natexlab{c}}} \textsc{Carrillo}, José~A.~;
  \textsc{Huang}, Yanghong~; \textsc{Schmidtchen}, Markus:
\newblock Zoology of a Nonlocal Cross-Diffusion Model for Two Species.
\newblock In: \emph{SIAM Journal on Applied Mathematics}
\newblock 78 (2018), Nr.~2, S.~1078--1104

\bibitem[Carrillo u.\,a.(2019{\natexlab{c}})Carrillo, Murakawa, Sato, Togashi
  und Trush]{CMSTT2019}
\dinatlabel{Carrillo u.\,a. 2019{\natexlab{c}}} \textsc{Carrillo}, José~A.~;
  \textsc{Murakawa}, Hideki~; \textsc{Sato}, Makoto~; \textsc{Togashi},
  Hideru~; \textsc{Trush}, Olena:
\newblock A population dynamics model of cell-cell adhesion incorporating
  population pressure and density saturation.
\newblock In: \emph{Journal of theoretical biology}
\newblock 474 (2019), S.~14--24

\bibitem[David(2022)]{Dav2022}
\dinatlabel{David 2022} \textsc{David}, Noemi:
\newblock Phenotypic heterogeneity in a model of tumor growth: existence of
  solutions and incompressible limit.
\newblock In: \emph{arXiv preprint arXiv:2204.05590}
\newblock (2022)

\bibitem[D{\k{e}}biec u.\,a.(2021)D{\k{e}}biec, Perthame, Schmidtchen und
  Vauchelet]{DPSV2021}
\dinatlabel{D{\k{e}}biec u.\,a. 2021} \textsc{D{\k{e}}biec}, Tomasz~;
  \textsc{Perthame}, Beno{\^\i}t~; \textsc{Schmidtchen}, Markus~;
  \textsc{Vauchelet}, Nicolas:
\newblock Incompressible limit for a two-species model with coupling through
  Brinkman's law in any dimension.
\newblock In: \emph{Journal de Math{\'e}matiques Pures et Appliqu{\'e}es}
\newblock 145 (2021), S.~204--239

\bibitem[D{\k{e}}biec und Schmidtchen(2020)]{DS2020}
\dinatlabel{D{\k{e}}biec und Schmidtchen 2020} \textsc{D{\k{e}}biec}, Tomasz~;
  \textsc{Schmidtchen}, Markus:
\newblock Incompressible limit for a two-species tumour model with coupling
  through {B}rinkman's law in one dimension.
\newblock In: \emph{Acta Applicandae Mathematicae}
\newblock 169 (2020), Nr.~1, S.~593--611

\bibitem[Degond und Mustieles(1990)]{DM1990}
\dinatlabel{Degond und Mustieles 1990} \textsc{Degond}, Pierre~;
  \textsc{Mustieles}, Francisco-Jos{\'e}:
\newblock A deterministic approximation of diffusion equations using particles.
\newblock In: \emph{SIAM Journal on Scientific and Statistical Computing}
\newblock 11 (1990), Nr.~2, S.~293--310

\bibitem[Di~Francesco u.\,a.(2018)Di~Francesco, Esposito und Fagioli]{DEF2018}
\dinatlabel{Di~Francesco u.\,a. 2018} \textsc{Di~Francesco}, Marco~;
  \textsc{Esposito}, Antonio~; \textsc{Fagioli}, Simone:
\newblock Nonlinear degenerate cross-diffusion systems with nonlocal
  interaction.
\newblock In: \emph{Nonlinear Analysis}
\newblock 169 (2018), S.~94--117

\bibitem[Ducasse u.\,a.(2023)Ducasse, Santambrogio und Yolda{\c{s}}]{DSY2023}
\dinatlabel{Ducasse u.\,a. 2023} \textsc{Ducasse}, Romain~;
  \textsc{Santambrogio}, Filippo~; \textsc{Yolda{\c{s}}}, Havva:
\newblock A cross-diffusion system obtained via (convex) relaxation in the JKO
  scheme.
\newblock In: \emph{Calculus of Variations and Partial Differential Equations}
\newblock 62 (2023), Nr.~1, S.~29

\bibitem[Galiano und Gonz{\'a}lez-Tabernero(2020)]{GG2020}
\dinatlabel{Galiano und Gonz{\'a}lez-Tabernero 2020} \textsc{Galiano},
  Gonzalo~; \textsc{Gonz{\'a}lez-Tabernero}, V{\'\i}ctor:
\newblock Turing instability analysis of a singular cross-diffusion problem.
\newblock In: \emph{Electronic Journal of Differential Equations}
\newblock 55 (2021), 1--17. 


\bibitem[Galiano und Selgas(2014)]{GS2014}
\dinatlabel{Galiano und Selgas 2014} \textsc{Galiano}, Gonzalo~;
  \textsc{Selgas}, Virginia:
\newblock On a cross-diffusion segregation problem arising from a model of
  interacting particles.
\newblock In: \emph{Nonlinear Analysis: Real World Applications}
\newblock 18 (2014), S.~34--49

\bibitem[Gurtin und Pipkin(1984)]{GP1984}
\dinatlabel{Gurtin und Pipkin 1984} \textsc{Gurtin}, Morton~E.~;
  \textsc{Pipkin}, AC:
\newblock A note on interacting populations that disperse to avoid crowding.
\newblock In: \emph{Quarterly of Applied Mathematics}
\newblock (1984), S.~87--94

\bibitem[Gwiazda u.\,a.(2019)Gwiazda, Perthame und
  {\'{S}wierczewska-Gwiazda}]{GPS2019}
\dinatlabel{Gwiazda u.\,a. 2019} \textsc{Gwiazda}, Piotr~; \textsc{Perthame},
  Beno\^{i}t~; \textsc{{\'{S}wierczewska-Gwiazda}}, Agnieszka:
\newblock A two-species hyperbolic--parabolic model of tissue growth.
\newblock In: \emph{Comm. Partial Differential Equations}
\newblock 44 (2019), Nr.~12, S.~1605--1618. --
\newblock ISSN 0360-5302

\bibitem[Jacobs(2021)]{Jac2021}
\dinatlabel{Jacobs 2021} \textsc{Jacobs}, Matt:
\newblock Existence of solutions to reaction cross diffusion systems.
\newblock In: \emph{ArXiv Preprint, arXiv:2107.12412}
\newblock (2021)

\bibitem[Jacobs(2022)]{Jac2022}
\dinatlabel{Jacobs 2022} \textsc{Jacobs}, Matt:
\newblock Non-mixing Lagrangian solutions to the multispecies Porous Media
  Equation.
\newblock In: \emph{ArXiv Preprint, arXiv:2208.01792}
\newblock (2022)

\bibitem[J{\"u}ngel(2015)]{Jun2015}
\dinatlabel{J{\"u}ngel 2015} \textsc{J{\"u}ngel}, Ansgar:
\newblock The boundedness-by-entropy method for cross-diffusion systems.
\newblock In: \emph{Nonlinearity}
\newblock 28 (2015), Nr.~6, S.~1963

\bibitem[Kim und M{\'e}sz{\'a}ros(2018)]{KM2018}
\dinatlabel{Kim und M{\'e}sz{\'a}ros 2018} \textsc{Kim}, Inwon~;
  \textsc{M{\'e}sz{\'a}ros}, Alp{\'a}r~R.:
\newblock On nonlinear cross-diffusion systems: an optimal transport approach.
\newblock In: \emph{Calculus of Variations and Partial Differential Equations}
\newblock 57 (2018), Apr, Nr.~3, S.~79. --
\newblock URL \url{https://doi.org/10.1007/s00526-018-1351-9}. --
\newblock ISSN 1432-0835

\bibitem[Lions und Mas-Gallic(2001)]{LM2001}
\dinatlabel{Lions und Mas-Gallic 2001} \textsc{Lions}, Pierre-Louis~;
  \textsc{Mas-Gallic}, Sylvie:
\newblock Une méthode particulaire déterministe pour des équations
  diffusives non linéaires.
\newblock In: \emph{Comptes Rendus de l'Académie des Sciences - Series I -
  Mathematics}
\newblock 332 (2001), Nr.~4, S.~369--376. --
\newblock URL
  \url{https://www.sciencedirect.com/science/article/pii/S076444420001795X}. --
\newblock ISSN 0764-4442

\bibitem[Liu und Xu(2021)]{LX2021}
\dinatlabel{Liu und Xu 2021} \textsc{Liu}, Jian-Guo~; \textsc{Xu}, Xiangsheng:
\newblock Existence and incompressible limit of a tissue growth model with
  autophagy.
\newblock In: \emph{SIAM Journal on Mathematical Analysis}
\newblock 53 (2021), Nr.~5

\bibitem[Mogilner und Edelstein-Keshet(1999)]{ME1999}
\dinatlabel{Mogilner und Edelstein-Keshet 1999} \textsc{Mogilner}, Alexander~;
  \textsc{Edelstein-Keshet}, Leah:
\newblock A non-local model for a swarm.
\newblock In: \emph{Journal of Mathematical Biology}
\newblock 38 (1999), Nr.~6, S.~534--570

\bibitem[Morisita(1950)]{Mor1950}
\dinatlabel{Morisita 1950} \textsc{Morisita}, M:
\newblock Population density and dispersal of a water strider. Gerris
  lacustris: Observations and considerations on animal aggregations.
\newblock In: \emph{Contributions on Physiology and Ecology, Kyoto University}
\newblock 65 (1950), S.~1--149

\bibitem[Morisita(1954)]{Mor1954}
\dinatlabel{Morisita 1954} \textsc{Morisita}, M:
\newblock Dispersion and population pressure: experimental studies on the
  population density of an ant-lion, Glenuroides japonicus m’l (2).
\newblock In: \emph{Jap. J. Ecol}
\newblock 4 (1954), Nr.~71, S.~9

\bibitem[Murakawa und Togashi(2015)]{MT2015}
\dinatlabel{Murakawa und Togashi 2015} \textsc{Murakawa}, Hideki~;
  \textsc{Togashi}, Hideru:
\newblock Continuous models for cell--cell adhesion.
\newblock In: \emph{Journal of theoretical biology}
\newblock 374 (2015), S.~1--12

\bibitem[Naldi u.\,a.(2010)Naldi, Pareschi und Toscani]{NPT2010}
\dinatlabel{Naldi u.\,a. 2010} \textsc{Naldi}, Giovanni~; \textsc{Pareschi},
  Lorenzo~; \textsc{Toscani}, Giuseppe:
\newblock \emph{Mathematical modeling of collective behavior in socio-economic
  and life sciences}.
\newblock Springer Science \& Business Media, 2010

\bibitem[Perthame u.\,a.(2014)Perthame, Quir\'{o}s und V\'{a}zquez]{PQV14}
\dinatlabel{Perthame u.\,a. 2014} \textsc{Perthame}, Beno\^{i}t~;
  \textsc{Quir\'{o}s}, Fernando~; \textsc{V\'{a}zquez}, Juan~L.:
\newblock The {H}ele-{S}haw asymptotics for mechanical models of tumor growth.
\newblock In: \emph{Arch. Ration. Mech. Anal.}
\newblock 212 (2014), Nr.~1, S.~93--127. --
\newblock ISSN 0003-9527

\bibitem[Perthame und Vauchelet(2015)]{PV2015}
\dinatlabel{Perthame und Vauchelet 2015} \textsc{Perthame}, Beno\^{i}t~;
  \textsc{Vauchelet}, Nicolas:
\newblock Incompressible limit of a mechanical model of tumour growth with
  viscosity.
\newblock In: \emph{Philos. Trans. Roy. Soc. A}
\newblock 373 (2015), Nr.~2050, S.~20140283, 16. --
\newblock URL \url{https://doi.org/10.1098/rsta.2014.0283}. --
\newblock ISSN 1364-503X

\bibitem[Price und Xu(2020)]{PX2020}
\dinatlabel{Price und Xu 2020} \textsc{Price}, Brock~C.~; \textsc{Xu},
  Xiangsheng:
\newblock Global existence theorem for a model governing the motion of two cell
  populations.
\newblock In: \emph{Kinetic and Related Models}
\newblock 13 (2020), Nr.~6, S.~1175--1191

\bibitem[Ranft u.\,a.(2010)Ranft, Basan, Elgeti, Joanny, Prost und
  J{\"u}licher]{RBEJPJ10}
\dinatlabel{Ranft u.\,a. 2010} \textsc{Ranft}, Jonas~; \textsc{Basan}, Markus~;
  \textsc{Elgeti}, Jens~; \textsc{Joanny}, Jean-Fran{\c{c}}ois~;
  \textsc{Prost}, Jacques~; \textsc{J{\"u}licher}, Frank:
\newblock Fluidization of tissues by cell division and apoptosis.
\newblock In: \emph{Proceedings of the National Academy of Sciences}
\newblock 107 (2010), Nr.~49, S.~20863--20868

\bibitem[Volkening und Sandstede(2015)]{VS2015}
\dinatlabel{Volkening und Sandstede 2015} \textsc{Volkening}, Alexandria~;
  \textsc{Sandstede}, Bj{\"o}rn:
\newblock Modelling stripe formation in zebrafish: an agent-based approach.
\newblock In: \emph{Journal of The Royal Society Interface}
\newblock 12 (2015), Nr.~112. --
\newblock URL
  \url{http://rsif.royalsocietypublishing.org/content/12/112/20150812}. --
\newblock ISSN 1742-5689

\end{thebibliography}

\end{document}